\documentclass[paper=a4, fontsize=12pt]{scrartcl}	

\usepackage[T1]{fontenc}
\usepackage[utf8]{inputenc}
\usepackage{lmodern}
\usepackage{bm}

\DeclareMathAlphabet{\mathsfit}{\encodingdefault}{\sfdefault}{m}{sl}
\SetMathAlphabet{\mathsfit}{bold}{\encodingdefault}{\sfdefault}{bx}{sl}

\newcommand{\tens}[1]{\bm{\mathsfit{#1}}}
\usepackage[english]{babel}								
\usepackage[protrusion=true,expansion=true]{microtype}			
\usepackage{amsmath}
\usepackage{amsfonts,amssymb}
\usepackage{amsthm}	
\usepackage{amsbsy}									
\usepackage[pdftex]{graphicx}								
\usepackage{color}
\usepackage{transparent}									
\usepackage[hang, small,labelfont=bf,up,textfont=it,up]{caption}	
\usepackage{epstopdf}									
\usepackage{subfig}										
\usepackage{booktabs}									
\usepackage{shuffle}										

\usepackage{verbatim}      
\usepackage{fancyhdr}
\numberwithin{equation}{section}						
\numberwithin{figure}{section}							
\numberwithin{table}{section}							

\newtheorem{thm}{Theorem}  
\newtheorem{prop}{Proposition}
\newtheorem{cor}{Corollary}
\newtheorem{lem}{Lemma}
\newtheorem{rem}{Remark}
\newtheorem{defn}{Definition}

\newcounter{example}[section]

\newcommand{\ck}{\color{black}}

\newcommand{\N}{\mathbb N}

\newcommand{\R}{\mathbb R}

\newcommand{\X}{\mathfrak{X}}

\newcommand{\Cc}{\mathcal{C}}
\newcommand{\Dd}{\mathcal{D}}
\newcommand{\Gg}{\mathcal{G}}

\newcommand{\CC}{\mathsf{C}}
\newcommand{\spn}{\mathsf{Span}}

\newcommand{\rk}{\mathsf{rank~}}	

\newcommand{\ad}{\mathrm{ad}}

\newcommand{\dd}{\mathrm{d}}
\newcommand{\im}{\mathrm{im}}

\newcommand{\TT}{\mathrm{T}}

\newcommand{\ddt}{\frac{\dd}{\dd t}}		

\newcommand{\ol}[1]{\overline{#1}}

\newcommand{\ds}{\displaystyle}
\newcommand{\ie}{\emph{i.e.}~}
\newcommand{\eg}{\emph{e.g.}~}

\newcommand{\apost}{\emph{a~posteriori}~}

\title{On Singularities of Flat Affine Systems With $\tens{n}$ States and $\tens{n}\bm{-1}$ Controls}

\author{Yirmeyahu J. Kaminski\thanks{Department of Applied
    Mathematics, Holon Institute of Technology, Holon, Israel, \hfill
    \break e-mail: jykaminski@gmail.com} \and Jean
  L\'{e}vine\thanks{CAS, Unit\'e Maths et Syst\`emes, MINES-ParisTech,
    PSL Research University, 60 Bd Saint-Michel, 75272 Paris Cedex 06,
    France, e-mail: jean.levine@mines-paristech.fr} \and Fran\c{c}ois
  Ollivier\thanks{CNRS, LIX laboratory, \'Ecole polytechnique, 91128 Palaiseau Cedex, France, \hfill \break e-mail: Francois.Ollivier@lix.polytechnique.fr}
}

\date{}

\begin{document}

\maketitle

\begin{abstract} We study the set of intrinsic singularities of flat affine systems with $n-1$ controls and $n$ states using the notion of Lie-B\"acklund atlas, previously introduced by the authors. For this purpose, we prove two easily computable sufficient conditions to construct flat outputs as a set of independent first integrals of distributions of vector fields, the first one in a generic case, namely in a neighborhood of a point where the $n-1$ control vector fields are independent, and the second one at a degenerate point where $p-1$ control vector fields are dependent of the $n-p$ others, with $p>1$.
After introducing the $\Gamma$-accessibility rank condition, we show
that the set of intrinsic singularities includes the set of points
where the system does not satisfy this rank condition and is included
in the set where a distribution of vector fields introduced in the
generic case is singular. We conclude this analysis by three examples
of apparent singularities of flat systems in generic and non generic
degenerate cases.
\end{abstract}

\section{Introduction}

Differential flatness \cite{FLMR95,FLMR99,Levine09} is known to be a
powerful notion in control theory. Roughly speaking, a system with $m$
independent controls $u=(u_{1},\ldots, u_{m})\in \R^{m}$ and state $x=
(x_{1}, \ldots, x_{n})$ defined on a $n$-dimensional smooth manifold
$X$, is said to be (differentially) flat at a given point $(x, u, \dot{u}, \ddot{u}, \ldots)$ of the infinite dimensional jet manifold
$$X\times \R^{m}_{\infty}  \triangleq X\times \R^{m} \times \R^{m}\times \cdots$$ 
if, and only if, its trajectories may be expressed, in
a neighborhood of this point, as functions of the trajectories of $m$ functionally independent smooth
variables, called flat outputs, and a finite number of their time
derivatives.  

Although non generic from a mathematical standpoint, flatness is a property shared by many popular models in various branches of engineering and has been shown to be particularly useful to solve motion planning problems (see \eg \cite{Levine09}).

In many cases, the flat outputs
can only be defined in a dense open set, and one may need to use
different parameterizations to cover the largest possible subset of the system configuration space, thus defining an atlas (see \cite{CE14,CE17,KLO17}).
Therefore, obtaining local flatness criteria allowing to build atlases covering the
\emph{widest possible domain} is an important issue, in particular since the complementary of this domain, by definition, is equal to the \emph{set of intrinsic singularities} \cite{KLO17}. It is remarkable that intrinsic flatness singularities may be interpreted as points in a neighbourhood of which the flatness-based control design is non robust since flat outputs stop existing there. On the contrary, an apparent singularity may allow a locally robust design by a suitable change of flat output.

In this paper, we continue our study of flatness singularities, initiated in \cite{KLO17}, by
restricting ourselves to control affine systems with $n$ states and $n-1$ controls, of the
form
$$\dot{x}= f_{0}(x) + \sum_{j=1}^{n-1} u_{j} f_{j}(x) \triangleq g(x,u).$$ 

After a brief recall, in this context, of concepts and notations related to flat systems and their singularities \cite{KLO17}, we focus attention on the $\Gamma$-accessibility rank condition at a point, the restriction at a point of the criterion studied, \eg\negthickspace, in \cite{HK77,JS72,Isidori86}, and the first order controllability around an integral curve generated by a constant control passing through this point, in the spirit of \eg \cite{SM67,K}, which are proven to be  necessary conditions for flatness\footnote{Note that, to the authors knowledge, and in spite of a long standing study of nonlinear controllability by many authors, these particular results are not available in the literature.}.  Therefore the points that do not satisfy them are naturally excluded from the above mentioned atlases and are thus contained in the so-called \emph{intrinsic singularity set}. Then, we show in theorem~\ref{theo::flat_outputs} that the sufficient condition for $\Gamma$-accessibility
$$\dim(\spn\{f_{1}, \ldots, f_{n-1},[g,f_{k}]\}) = n$$ 
for some $k$, is also a sufficient condition for flatness. Moreover, we prove that flat outputs can be obtained as \emph{independent first integrals} of the above field
$f_{k}$ in each neighborhood where the condition holds: this is a consequence of the equivalence to the Brunovsk\`{y} controllability canonical form where at least one of the controls, $u_k$, cannot appear in the first order derivatives of all components of the diffeomorphism but one, these $n-1$ components being therefore first integrals of the corresponding vector field $f_k$. 
This construction sheds a new light on a comparable result by Ph. Martin \cite{Martin93}, obtained by input-output  and structure at infinity considerations. 

Since the points where the above condition holds are defined by the independence of $n$ vectors, they are naturally said \emph{generic}. 

Again using a result of Ph. Martin \cite{MartinThese} clarifying the relationship between Lie-B\"{a}cklund isomorphism and dynamic feedback, we then show that our result can be interpreted in terms of (extended state) feedback linearization \cite{JakubczykRespondek80, Huntetal83} and draw some consequences on the set of intrinsic singularities. 

Let us also mention some complementary approaches, indirectly related to the present singularity study, on the existence of flat outputs in a neighborhood of a point, under assumptions on the differential weight or the prolongation length,  \eg, in  \cite{Resp-2003,NicolauThese,Nic-Resp-siam2017}. 

We proceed in our singularity study with the following question: \emph{Are there points of the state space where the dimension of the vector space generated by the control vector fields $f_{1}, \ldots, f_{n-1}$ drops down that are nevertheless apparent singularities?} Note that such points, if they exist, may be called \emph{non generic} for obvious reasons. We prove a theorem giving a new construction of flat outputs at such points, thus providing a positive answer to the previous question.

The above mentioned theorems and their consequences on flatness singularities constitute the main results of this paper. Though the first theorem, that gives a sufficient condition for flatness at generic points, was already known in a different perspective, our approach is here completely renewed compared to our previous paper \cite{KLO17} since it deals with distributions of vector fields for systems represented by explicit differential equations and provides direct and computable flat output constructions by first integrals, as can be seen in the three examples at the end of this paper. It is also interesting to remark that the result at non generic points (theorem~\ref{flat-non-gener-thm}) is still valid, as is, with $m\leq n-1$ inputs, as illustrated by the third example. We may also stress that, in these results, the singularities are not given in terms of singularities of the parameterization as in \cite{CE14,CE17,KLO17} but rather as singularities of distributions of vector fields.

The paper is organized as follows: the basics of control affine systems with $n$ states and $n-1$ controls as well as those on flatness singularities are recalled in sections~\ref{sec::model} and~\ref{sec:atlas}. Section~\ref{sec:intrinsic-gen} is then devoted to the singularity study at generic and non generic points. Three academic examples are then presented in section~\ref{sec:ex}, finally followed by concluding remarks in section~\ref{sec:concl}.

\section{Control Affine Systems with $\tens{n}$ States and $\tens{n}\bm{-1}$ Inputs}
\label{sec::model}

We consider a control affine system with drift given, in a local chart, by: 
\begin{equation}
\label{equ::explicit}
\dot{x} = f_{0}(x) + \sum_{j=1}^{n-1} u_{j} f_{j}(x)
\end{equation}
where the state $x$ evolves in a manifold $X$ of dimension $n \geq 2$,
with drift $f_{0}$, and with $m=n-1$ independent controls. 

We also make the following classical assumption:

The vector fields\footnote{The notations $f_{i}$, or $g$, etc. may be indifferently understood as usual vectors in prescribed local coordinates $(x_{1}, \ldots, x_{n})$, \ie 
$\left( f_{i,1},\ldots,f_{i,n}\right)^{T}$, or $\left( g_{1},\ldots, g_{n}\right)^{T}$, etc., the superscript $^{T}$ standing for transpose, or as the associated first order partial differential operators, \ie $\sum_{j=1}^{n}f_{i,j}\frac{\partial}{\partial x_{j}}$, or $\sum_{j=1}^{n}g_{j}\frac{\partial}{\partial x_{j}}$, etc.} $f_1, \cdots, f_{n-1}$ are assumed to be $C^\infty$
and linearly independent in a dense open set of $X$.  

In other words, there is a dense open set where the matrix 
\begin{equation}\label{Gdef:eq}
G(x) \triangleq \left(\begin{array}{ccc}f_{1}&\cdots&f_{n-1}\end{array}\right),
\end{equation} 
of size  $n\times (n-1)$, has full rank.


In the sequel, for simplicity's sake, we denote by $g$ the vector field in \eqref{equ::explicit}, pointwise defined by
\begin{equation}\label{gdef:eq}
g(x,u) \triangleq f_{0}(x) + \sum_{i=1}^{n-1} u_{i} f_{i}(x).
\end{equation}

\section{Recalls on the Infinite Order Jets Approach to Flat Systems with $\tens{n}\bm{-1}$ Inputs and Their Singularities }
\label{sec:atlas}

In this section, 
we briefly recall and adapt the
main background and tools, introduced and defined
in~\cite{KLO17}, to the
present context of systems with $n-1$ inputs. 

\subsection{The Formalism of Infinite Order Jets}

The definition of flatness introduced in \cite{FLMR99} 
requires the use of infinite order jets. More precisely, we
embed the manifold $X$ and the associated system~\eqref{equ::explicit} in the manifold  
$$\X \triangleq  X\times \R^{n-1}_{\infty} = X\times \R^{n-1} \times
\R^{n-1} \times \cdots$$ 
with coordinates 
$$(x,\overline{u}) \triangleq (x,u^{(0)},u^{(1)},u^{(2)},\ldots, u^{(k)}, \ldots ),$$  
endowed with the product topology.

In this topology, a  continuous (resp. differentiable) function from
$X \times \R^{n-1}_{\infty}$ to $\R$, by construction, \textit{only depends on a finite number of coordinates} and is continuous (resp. differentiable) with respect to these coordinates in the usual (finite dimensional)
sense. 

$\X =  X\times \R^{n-1}_{\infty}$ is also endowed with the Cartan vector field
\begin{equation}
\label{eq::cartan_field}
C_g = \sum_{i=1}^{n}g_{i}(x,\overline{u}) \frac{\partial}{\partial
  x_{i}} + \sum_{i=1}^{n-1} \sum_{j \geq 0} u_i^{(j+1)}
\frac{\partial}{\partial u_i^{(j)}}  \triangleq g(x,\overline{u}) \frac{\partial}{\partial x} + \sum_{j \geq 0} u^{(j+1)}
\frac{\partial}{\partial u^{(j)}}
\end{equation}
with $g$ defined by \eqref{gdef:eq}.

Considering $C_g$ as a first order differential operator and $h: \X\mapsto \R$ an arbitrary differentiable function,  interpreting the expression
$C_g h = \sum_{i=1}^{n}g_{i}(x,\overline{u}) \frac{\partial h}{\partial
  x_{i}} + \sum_{i=1}^{n-1} \sum_{j \geq 0} u_i^{(j+1)}
\frac{\partial h}{\partial u_i^{(j)}}$,
as the Lie derivative of $h$ along the vector field $g$ of $\TT\X$, the tangent bundle of $\X$,  this amounts to identify $C_g $ with the vector $\left( g, \dot{u}, \ddot{u}\ldots\right)$ and equation \eqref{equ::explicit} with the infinite
number of equations 
$$\dot{x} = g(x,\overline{u}), \quad \dot{u}^{(0)} = u^{(1)}, \quad  \ldots,
\quad \dot{u}^{(k)}=u^{(k+1)} ,\quad \ldots $$

\subsection{Lie-B\"acklund Equivalence}\label{L-B-equiv:sec}

Consider two systems:

\begin{equation}\label{sys-equiv:def}
\dot{x}=g(x,u) \quad \mathrm{and} \quad \dot{y}=\gamma(y,v)
\end{equation}
and their prolongations on $X\times \R^{n-1}_{\infty}$ and $Y\times
\R^{\mu}_{\infty}$ respectively 
with the associated Cartan fields, in condensed notations introduced in~\eqref{eq::cartan_field}:

\begin{equation}\label{sys-equiv-C:def}
C_{g}= g(x,u)\frac{\partial}{\partial x} + \sum_{j\geq 0}
  u^{(j+1)}\frac{\partial}{\partial u^{(j)}}, 
\quad
C_{\gamma}= \gamma(y,v)\frac{\partial}{\partial y} + \sum_{j\geq 0} v^{(j+1)}\frac{\partial}{\partial v^{(j)}}
\end{equation}

We say that they are  \emph{Lie-B\"acklund equivalent} at a pair of points $(x_0,\ol{u}_0)$ and $(y_0,\ol{v}_0)$ if there exist neighborhoods of these points where every integral curve of one is mapped into an integral curve of the other and conversely. 

In other words, the two systems are \emph{Lie-B\"acklund equivalent} at the points $(x_0,\ol{u}_0)$ and $(y_0,\ol{v}_0)$  if there exists neighborhoods ${\mathcal N}_{x_{0},\ol{u}_{0}} \subset X\times \R^{n-1}_{\infty}$ and ${\mathcal N}_{y_{0},\ol{v}_{0}} \subset Y\times \R^{\mu}_{\infty}$ and a $C^{\infty}$ isomorphism $\Phi : {\mathcal N}_{y_{0},\ol{v}_{0}} \rightarrow {\mathcal N}_{x_{0},\ol{u}_{0}}$ satisfying $\Phi(y_{0},\ol{v}_{0})= (x_0,\ol{u}_0)$, with $C^{\infty}$ inverse  $\Psi$, such that the respective Cartan fields are $\Phi$ and $\Psi$ related, \ie $\Phi_{\ast}C_{\gamma}=C_{g}$ in ${\mathcal N}_{x_{0},\ol{u}_{0}}$ and $\Psi_{\ast}C_{g}=C_{\gamma}$ in ${\mathcal N}_{y_{0},\ol{v}_{0}}$.

We recall, without proof, a most important result from \cite{MartinThese} (see also \cite{FLMR95,FLMR99,Levine09}) giving an interpretation of the Lie-B\"{a}cklund equivalence in terms of diffeomorphism and endogenous dynamic feedback, that will be useful in the next sections. We state it in the present context of systems with $n-1$ inputs for convenience, though the result is much more general.

\begin{thm}[Martin~\cite{MartinThese}]\label{endo-equiv:thm}
If the systems \eqref{sys-equiv:def} are Lie-B\"{a}cklund equivalent at a given pair of points, then (i) and (ii) must be satisfied:
\begin{itemize}
\item[(i)] $n-1=\mu$, \ie they must have the same number of independent inputs;
\item[(ii)]
there exist 
\begin{itemize}
\item an endogenous dynamic feedback\footnote{A dynamic feedback is said endogenous if, and only if, the closed-loop system and the original one are Lie-B\"{a}cklund equivalent, \ie if, and only if, the extended state $z$ can be locally expressed as a smooth function of $x$, $u$ and a finite number of time derivatives of $u$ (see \cite{MartinThese,FLMR95,FLMR99,Levine09}).} 
\begin{equation}\label{dynfeed:def}
u=\alpha(x,z,w), \quad  \dot{z}=\beta(x,z,w),
 \end{equation} 
\item a multi-integer\footnote{Recall that we denote by 
$v^{(r)} \triangleq \left( v_{1}^{(r_{1})}, \ldots, v_{n-1}^{(r_{n-1})} \right)= 
\left( \frac{\dd^{r_{1}}v_{1}}{\dd t^{r_{1}}}, \ldots, \frac{\dd^{r_{n-1}}v_{n-1}}{\dd t^{r_{n-1}}}\right)$.} 
$r \triangleq \left( r_{1}, \ldots, r_{n-1}\right)$,
\item and a local diffeomorphism  $\chi$, 
\end{itemize}
all defined in a neighborhood of the considered points, such that the closed-loop system 
\begin{equation}\label{CLdynfeed:eq}
\dot{x}=g(x,\alpha(x,z,w)), \quad \dot{z}=\beta(x,z,w)
\end{equation} 
is locally diffeomorphic to the extended one
\begin{equation}\label{CLdynfeed-equiv:eq}
\dot{y}= \gamma(y,v), \quad v^{(r)}=w
\end{equation} 
for all $w\in \R^{n-1}$, \ie 
\begin{equation}\label{extended-diffeo:eq}
(x,z)=\chi(y,v, \dot{v},\ldots,v^{(r-1)}), \qquad 
(y,v, \dot{v},\ldots,v^{(r-1)})=\chi^{-1}(x,z)
\end{equation}
and 
\begin{equation}\label{extended-diffeo-fields:eq}
\hat{g} =\chi_{\ast}\hat{\gamma}, \qquad
\hat{\gamma} = \chi^{-1}_{\ast}\hat{g}
\end{equation}
where we have denoted
$$
\begin{aligned}
&\hat{g}(x,z,w) \triangleq g(x,\alpha(x,z,w))\frac{\partial}{\partial x} + \beta(x,z,w)\frac{\partial}{\partial z}\\
&\hat{\gamma}(y,v, \dot{v}, \ldots, v^{(r-1)},w) \triangleq \gamma(y,v)\frac{\partial}{\partial y} + \sum_{j=0}^{r-2} v^{(j+1)}\frac{\partial}{\partial v^{(j)}} + w\frac{\partial}{\partial v^{(r-1)}}.
\end{aligned}
$$
\end{itemize}
\end{thm}

\subsection{Flatness}

We say that system  \eqref{equ::explicit} is \emph{differentially flat} (or, more shortly, \emph{flat}) at the pair of points $(x_0,\ol{u}_0)$ and $\ol{y}_{0}$ if and only if, it is \emph{Lie-B\"{a}cklund equivalent to the trivial system $\R^{n-1}_{\infty}$ endowed with the trivial Cartan field
$$\tau =
\sum_{j\geq 0}\sum_{i=1}^{n-1} u_{i}^{(j+1)}\frac{\partial}{\partial u_{i}^{(j)}}$$
at the considered points}.

Otherwise stated, the locally defined flat output $y=\Psi(x, \ol{u})$ is such that $(x,\ol{u})= \Phi(\ol{y}) = (\Phi_0(\ol{y}), \Phi_1(\ol{y}), \Phi_2(\ol{y}), \ldots)$ with $\ddt{\Phi}_0(\ol{y})= g(\Phi_0(\ol{y}), \Phi_1(\ol{y}))$ for all sufficiently differentiable  $y$.

This definition immediately implies that a system is flat if there exists a generalized output $y= \Psi(x,\ol{u})$ of dimension $n-1$, thus depending at most on a finite number of derivatives of $u$, with independent derivatives of all orders, such that $x$ and $\ol{u}$ can be expressed in terms of $y$ and a finite number of successive derivatives, \ie $(x,\ol{u})= \Phi(\ol{y})$, and such that the system equation $\ddt{\Phi_{0}}(\ol{y})= g(\Phi(\ol{y}))$ is identically satisfied for all sufficiently differentiable $y$.

For a flat system, with the notations of subsection~\ref{L-B-equiv:sec}, the vector field $\gamma$, or $\widehat{\gamma}$ indifferently, corresponds  to the linear system in Brunovsk\'{y} canonical form
\begin{equation}\label{flat-CLdynfeed-equiv:eq}
y_{i}^{(r_{i}+1)}=w_{i}, \qquad i=1,\ldots,n-1,
\end{equation}
 $C_{\gamma} = \tau$ (with global coordinates $\ol{y}\triangleq (y,\dot{y},\ldots)$ in place of $\ol{u}$) and theorem~\ref{endo-equiv:thm} reads:
\begin{cor}
If system \eqref{equ::explicit}, with notation \eqref{gdef:eq}, is flat at a given point, it is dynamic feedback linearizable in a neighborhood of this point, \ie there exists an endogenous dynamic feedback of the form \eqref{dynfeed:def} and a local diffeomorphism $\chi$ such that the closed-loop system \eqref{CLdynfeed:eq} is transformed by $\chi$ into \eqref{flat-CLdynfeed-equiv:eq} for all $w\in \R^{n-1}$.
\end{cor}

\subsection{Lie-B\"acklund Atlas}

The notion of a Lie-B\"acklund atlas for flat systems was initially introduced in~\cite{KLO17} in the context of implicit systems. Our presentation here adapts this definition to the case of systems in explicit form. It consists of a collection of charts on $\X$, that we call \emph{Lie-B\"acklund charts and atlas}, and that will allow us to define the notions of apparent and intrinsic singularities.

\begin{defn}\label{L-B-charts:def}
\begin{itemize}
\item[(i)] A Lie-B\"acklund chart on $\X$ is the data of a pair $(\mathfrak{U}, \psi)$ where $\mathfrak{U}$ is an open set of $\X$ and $\psi : \mathfrak{U} \rightarrow \R^{n-1}_{\infty}$ a local flat output,  with local inverse $\varphi: \mathfrak{V}\rightarrow \mathfrak{U}$, $ \mathfrak{V}\triangleq \psi(\mathfrak{U})$ being an open set of $\R^{n-1}_{\infty}$.
\item[(ii)] Two charts $(\mathfrak{U}_1,\psi_1)$ and $(\mathfrak{U}_2,\psi_2)$ are said to be compatible if, and only if, the mapping
$$\psi_{1}\circ \varphi_{2}: \psi_{2}(\varphi_{1}(\mathfrak{V}_1)\cap \varphi_2(\mathfrak{V}_2) ) \subset \R^{n-1}_{\infty} \rightarrow \psi_{1}(\varphi_{1}(\mathfrak{V}_1)\cap \varphi_2(\mathfrak{V}_2) ) \subset \R^{n-1}_{\infty}$$ 
with $\mathfrak{V}_i = \psi_{i}(\mathfrak{U}_i)$, $i=1,2$, is a local Lie-B\"acklund isomorphism (with the same trivial Cartan field $\tau$ associated to both the source and the target) with local inverse $\psi_{2}\circ \varphi_{1}$, as long as
$\varphi_{1}(\mathfrak{V}_1)\cap \varphi_2(\mathfrak{V}_2) \neq \emptyset$.

Note that charts with nonempty intersection are always compatible since the composition of Lie-B\"{a}cklund isomorphisms is a Lie-B\"{a}cklund isomorphism, or, otherwise stated, in reason of the transitivity of the Lie-B\"{a}cklund equivalence relation.
\item[(iii)] An atlas $\mathfrak{A}$ is a collection of compatible charts. 
\end{itemize}
\end{defn}

For a given atlas $\mathfrak{A} = (\mathfrak{U}_i,\psi_i)_{i \in I}$, let $\mathfrak{U}_\mathfrak{A}$ be the union $\mathfrak{U}_\mathfrak{A} = \bigcup_{i \in I} \mathfrak{U}_i$. 

\begin{rem}\label{charts:rem}
In definition \ref{L-B-charts:def}, we stress that Lie-B\"acklund isomorphisms play a similar role as the smooth diffeomorphisms appearing in the definition of a usual smooth manifold, at the exception that we do not require that $\mathfrak{U}_\mathfrak{A} = \X$. 
\end{rem}

\begin{rem}\label{triv-topo:rem}
The charts are made of open sets that are homeomorphic to open sets of $\left( \R^{n-1} \right)^N$ for some finite $N$, according to (i), and thus it is always possible to construct a larger atlas entirely made of topologically trivial (\ie contractible) open sets.
\end{rem}

\subsection{Apparent and Intrinsic Flatness Singularities}

It is clear from what precedes that if we are given two Lie-B\"acklund atlases, their union is again a Lie-B\"acklund atlas. Therefore the union of all charts that form every atlas is well-defined as well as its complementary, which we call the set of intrinsic flatness singularities, as stated in the next definition.

\begin{defn}
\label{def::intrinsic-singularities}
We say that a point in $\X$ is an intrinsic flatness singularity if it is excluded from all charts of every Lie-B\"acklund atlas. Every other singular point, namely every point $\ol{x}\not\in \mathfrak{U}_i$ for some chart $(\mathfrak{U}_i,\psi_i)$ but for which there exists another chart $(\mathfrak{U}_j,\psi_j)$, $j\neq i$, such that $\ol{x}\in \mathfrak{U}_j$, is called apparent.
\end{defn}

Clearly, this notion does not depend on the choice of atlas and charts. The concrete meaning of this notion is that at points that are intrinsic singularities there is no flat output, i.e. the system is not flat at these points. 

On the other hand, points that are apparent singularities are singular for a given set of flat outputs, but well defined points for another set of flat outputs defined in another chart containing these points.

\section{Intrinsic Flatness Singularity Study of Control Affine Systems with $\tens{n}\bm{-1}$ Inputs}
\label{sec:intrinsic-gen}

\subsection{On Flat Output Computation and Lie-B\"acklund Atlas Construction at Generic Points}
\label{flatcomp-gen:subsec}

We start this section with the singularity study (and thus the flat output computation) at points $x\in X$ such that the vector space generated by the control vector fields $f_{1}(x), \ldots, f_{n-1}(x)$ has dimension equal to $n-1$ and remains constant in a suitable open neighborhood of $x$. We call these points \emph{generic} for obvious reasons. We first prove that these generic points are such that the $\Gamma$-accessibility rank condition (see \cite{HK77,JS72}) is satisfied, or equivalently that the first approximation of the system around germs of integral curves passing through these points is controllable (see \cite{SM67,K}), 
and give a first sufficient condition for the existence of flat outputs at these points. We then construct the associated Lie-B\"{a}cklund charts and atlas.

\subsubsection{$\Gamma$-accessibility Rank Condition of Affine Systems with $\tens{n}\bm{-1}$ Inputs}\label{SARC:subsubsec}

 Let us recall the classical Lie bracket notations: 
 $[\eta,\gamma]\triangleq \frac{\partial \gamma}{\partial x}\eta -\frac{\partial \eta}{\partial x}\gamma$ 
denotes the Lie bracket of the vector
fields $\eta$ and $\gamma$, and, iteratively, $\ad_{\eta}\gamma
\triangleq [\eta,\gamma]$, $\ad^{k}_{\eta}\gamma=[\eta,
  \ad^{k-1}_{\eta}\gamma]$, with $\ad^{0}_{\eta}\gamma=\gamma$, for
all $k\geq 0$. We also denote  the $n\times (n-1)$ matrix $\ad^{k}_{g} G$ by
\begin{equation}\label{adkgdef:eq}
\ad^{k}_{g} G \triangleq \left(
\begin{array}{ccc}\ad^{k}_{g} f_{1}&\cdots&\ad^{k}_{g} f_{n-1}\end{array}\right),
\end{equation} 
with $G$ defined by \eqref{Gdef:eq}, and the $n\times (k+1)(n-1)$ matrix $\Gg_k$, for all $k\geq 1$, by
\begin{equation}\label{GGdef:eq}
\Gg_k \triangleq \left( \begin{array}{cccc}
G&-\ad_{g} G&\cdots&(-1)^{k}\ad_{g}^{k} G
\end{array}
\right),
\end{equation} 
that may be interpreted as the Wronskian matrix of $G$ (see \cite{SM67}), where the successive time derivative operators $\frac{\dd^{k}}{\dd t^{k}}$ are replaced by the iterated Lie bracket operators $(-1)^{k}\ad_{g}^{k}$. 
%
Note that, for linear
systems, $\Gg_n$ is often called the Kalman controllability matrix (see
\eqref{brack-g-fk:eq} in the proof of
lemma~\ref{lemma::controllability_at_first_oder}).

\begin{defn}\label{def:gam_contr_rk}
Given the following sequence of distributions:
\begin{equation}\label{Gammak:def}
\Gamma_{0} \triangleq
\spn\{f_1, \cdots, f_{n-1}\}, \qquad 
\Gamma_{k+1} \triangleq \Gamma_{k} + \ad_{g}\Gamma_{k}, \quad k\geq 0,
\end{equation}
we denote by $\Gamma_{k}(x,u)= \{\gamma(x,u) : \gamma\in \Gamma_{k}\}$.
We say that the \emph{$\Gamma$-accessibility rank condition} is satisfied at the point $(x_{0},u_{0})$ if, and only if, there exists $k^{\star} \in \N$ such that $\Gamma_{k^{\star}}(x_{0},u_{0}) = \Gamma_{\infty}(x_{0},u_{0}) = \TT_{x_{0},u_{0}}X$, \ie if $\dim\Gamma_{k^{\star}}(x_{0},u_{0}) = \dim \Gamma_{\infty}(x_{0},u_{0}) =n$.
\end{defn} 

\begin{rem}$\Gamma_{k^{\star}}$ is thus the Lie ideal generated by $f_1, \cdots, f_{n-1}$ in the Lie algebra generated by $g, f_1, \cdots, f_{n-1}$.
\end{rem}

\begin{defn}\label{def:strong_contr_rk}
Consider the other sequence of distributions
\begin{equation}\label{Dk-seq:def}
\Dd_{k+1} = \ol{\Dd}_{k} + \ad_{f_{0}} \ol{\Dd}_{k}, \quad k\geq 0, \quad \Dd_{0} = \Gamma_{0}
\end{equation}
where $\ol{\Dd}_{k}$ is the involutive closure of  $\Dd_{k}$ (see \eg \cite{Isidori86}).
The condition $\dim \Dd_{\infty}(x_{0}) = n$ is called the \emph{strong accessibility rank condition} \cite{HK77}.
\end{defn} 
\begin{prop}\label{prop:strong_contr_rk}
We have $\Gamma_{k}\subset \Dd_{k} $ for all $k$ and $u\in \R^{m}$ and $\Gamma_{\infty}(x_{0},u_{0}) = \Dd_{\infty}(x_{0})$ for all $u_{0}$ in a dense subset of $\R^{m}$. Moreover, $\dim \Gamma_{\infty} (x_{0},u_{0}) = n$ implies $\dim \Dd_{\infty} (x_{0}) = n$. In other words, 
the $\Gamma$-accessibility rank condition implies the strong accessibility rank condition.
\end{prop}
\begin{proof}
Since $g= f_{0}+\sum_{i=1}^{n-1}u_{i}f_{i}$, we indeed have $\Gamma_{k}\subset \Dd_{k} $ for all $k$ and $u\in \R^{m}$ and $\Gamma_{\infty}(x_{0},u_{0}) = \Dd_{\infty}(x_{0})$ for all $u_{0}$ in a dense subset of $\R^{m}$.
 Again, since $\Gamma_{k}\subset \Dd_{k} $ for all $k$ and $u\in \R^{m}$, we immediately deduce that $\Gamma$-accessibility implies strong accessibility.
\end{proof}
\begin{rem}
If the distributions \eqref{Gammak:def} have constant rank in a neighborhood of $(x_{0},u_{0})$, it can be readily verified that $k^{\star}$, defined by definition~\ref{def:gam_contr_rk}, satisfies $k^{\star} \leq n-1$. If, on the contrary, the rank of some of these distributions drops down at $(x_{0},u_{0})$, then it is possible that $k^{\star} > n$.
\end{rem}

\begin{thm}\label{access_rk_cond:thm}
Every flat system at a point satisfies the $\Gamma$-accessibility rank condition (and thus the strong accessibility rank condition) at this point.
\end{thm}
\begin{proof}
Given a flat system at a point, we consider the associated equivalent linear system~\eqref{flat-CLdynfeed-equiv:eq} which is indeed controllable. The Kalman controllability matrix of this system has constant rank $n' \triangleq \sum_{i=1}^{n-1} r_{i}$, with the notations of theorem~\ref{endo-equiv:thm}. As already noted, to this matrix, there corresponds the increasing sequence of distributions generated by the control vector fields $\hat{\Gamma}_{0}\triangleq \spn \{\chi_{\ast}\left( \frac{\partial}{\partial y_{i}^{(r_{i})}} \right), i=1, \ldots, n-1\}$ and their iterated Lie brackets with the vector field $\hat{g}$, defined by \eqref{extended-diffeo-fields:eq},  \ie 
$\hat{\Gamma}_{k+1} \triangleq \hat{\Gamma}_{k} + \ad_{\hat{g}}\hat{\Gamma}_{k}$, for $k\geq 0$,
analogously to the construction  \eqref{Gammak:def}. According to the properties of the image of the Lie bracket by diffeomorphism, it results that the distributions $\hat{\Gamma}_{k}$ must have locally a constant dimension and that the largest one, say $\hat{\Gamma}_{\infty}$, must have constant dimension equal to $n'$. We also note that since $\chi$ is a local diffeomorphism satisfying $(x,z) = \chi (y,\dot{y},\ldots, y^{(r-1)})$, we indeed have $n'\geq n$. We  thus assume that the $\Gamma$-accessibility rank condition is not satisfied or, equivalently, that $\dim\Gamma_{k^{\star}}(x_{0},u_{0}) < n$. Since, by construction, the projection on $\TT_{X}$ of $\hat{\Gamma}_{\infty}$ is contained in $\Gamma_{k^{\star}}$, there must exist at least a non 0 combination of the $x_{j}$'s, $j \in \{1, \ldots, n-1\}$, denoted by $\xi$, such that $d\xi\in \left(\hat{\Gamma}_{\infty}\right)^{\perp}$. But then an immediate computation shows that $d\chi^{-1}(\xi)$ must be independent of the inputs $y_{i}^{(r_{i})}$ for all $i= 1, \ldots, n-1$, which contradicts the controllability of system~\eqref{flat-CLdynfeed-equiv:eq}, hence the result.
\end{proof}

A simple interpretation of the $\Gamma$-accessibility rank condition may be given in terms of controllability of the first order time-varying linear approximation of the system:
\begin{defn}\label{def:1st-order_contr-traj}
System \eqref{equ::explicit} is said controllable at the
first order with constant input at $(x_{0},u_{0})\in X\times \R^{n-1}$, or first order controllable with constant input $u_{0}\in \R^{n-1}$, at a point $x_{0} \in X$ 
if, and only if, its tangent linear approximation along the integral curve $(x(t),u_{0})$ passing through $(x_0, u_0)$ at time $t=0$, and with constant input $u(t)=u_{0}$ for all $t$ in a small interval of time containing 0, namely
$\delta \dot{x} = \frac{\partial g}{\partial x}(x(t),u_{0})\delta x +
\sum_{i=1}^{n-1} f_{i}(x(t))\delta u_{i}$, is controllable
in the sense of linear time-varying systems (see e.g. \cite{SM67,K}).  
\end{defn}

\begin{lem}
\label{lemma::controllability_at_first_oder}
The system is controllable at the first order with constant input at $(x_0,u_0)$ if, and
only if, in an open neighborhood of $(x_0,u_0)$, there exists $k^{\star}\in \N$ such that: 
\begin{equation}\label{eq::gen_first_order_controllability}
\rk{\Gg_{k^{\star}}}(x_0,u_0)= n
\end{equation}
or, equivalently, if, and only if, the $\Gamma$-accessibility rank condition is satisfied at $(x_0,u_0)$.

Moreover, if in an open neighborhood of $(x_{0},u_{0})$ there exists
$k \in \{0,\ldots,n-1\}$ such that 
\begin{equation}
\label{eq::first_order_controllability}
\dim(\spn\{f_{1}, \ldots, f_{n-1},[g,f_{k}]\}) = n 
\end{equation}
then condition \eqref{eq::gen_first_order_controllability} is satisfied.
\end{lem}

\begin{proof}
The tangent linear system along the integral curve $t\mapsto (x(t),u_{0})$, with constant $u_{0}\triangleq \left( u_{1,0},\ldots, u_{n-1,0}\right)$ for $t$ in a given open interval $I$ containing $0$, is given by: $\delta \dot{x} = A(t) \delta x + B(t) \delta u$, where 
$$A(t) \triangleq \frac{\partial f_0}{\partial x}(x(t)) + \sum_{i=1}^{n-1} u_{i,0} \frac{\partial f_i}{\partial x}(x(t)), \qquad B(t) \triangleq G(x(t)).$$ 
According to \cite{SM67,K}, this linear time-varying system is controllable at $(x_{0},u_{0})$ if, and only if the controllability matrix 
$$\Cc(0) \triangleq \left( B(t),(A(t) - \ddt)B(t),\cdots,(A(t) - \ddt)^{k^{\star}}B(t)\right)_{\mid t=0}$$
 has rank $n$ for some $k^{\star}$. 

 On the other hand, we have, using \eqref{gdef:eq},  in matrix notation:
 $$\ad_{g} f_{k}=
 \frac{\partial f_{k}}{\partial x} 
  \left( 
 f_{0}+ \sum_{i=1}^{n-1} u_{i,0} f_{i}
 \right)
- 
\left( \frac{\partial f_{0}}{\partial x} + \sum_{i=1}^{n-1} u_{i,0}\frac{\partial f_{i}}{\partial x} \right)
f_{k} 
 $$

Therefore, an easy direct computation yields: 
\begin{equation}\label{brack-g-fk:eq}
\ad_{g} G = \left(\frac{\dd B}{\dd t} (t)- A(t)B(t)\right)_{\mid t=0}  = - \left(A(t) - \ddt\right)B(t)_{\mid t=0}.
\end{equation}
Thus, by induction, we get $\ad_{g}^k G =  - \left(A(t) - \frac{\dd}{\dd t} \right)\left( (-1)^{k-1} \left(A(t) - \frac{\dd}{\dd t} \right)^{k-1} B(t)\right) _{\mid t=0} = (-1)^k  \left( \left(A(t) - \ddt\right)^kB(t)\right)_{\mid t=0}$, since $u_{0}$ is constant, and the controllability matrix $\Cc(0)$ is proven to be equal to $\Gg_{k^{\star}}(x_0,u_0)$.

Moreover, assuming that, at $(x_{0},u_{0})$,  $\ad_{g} f_{k} \not \in  \spn\{f_1, \cdots, f_{n-1}\} = \im{~G}$ for some $k\in \{1,\ldots, n-1\}$, then the matrix $\left( \begin{array}{cc} G&-\ad_{g} G \end{array}\right)$ has rank $n$, which immediately implies that the controllability matrix $\Gg_{1}(x_0,u_0)=\Cc(0)$ has full rank, hence the first order controllability at $(x_0,u_0)$.
\end{proof}

\subsubsection{Flat Outputs for Affine Systems with $\tens{n}\bm{-1}$ Inputs at Generic Points}
\label{subsec:flat-gen}

\begin{thm}
\label{theo::flat_outputs}
Let $(x_0,u_0) \in X \times \R^{n-1}$ and $k$ be such that
assumption~\eqref{eq::first_order_controllability} holds in an open
neighborhood of $(x_{0},u_{0})$. Then, every $(n-1)$-tuple of first integrals of $f_k$, independent at
$(x_{0},u_{0})$, forms a vector of flat
outputs in an open neighborhood of $(x_{0},u_{0})$.
\end{thm}
\begin{proof}
We consider $f_k$ satisfying assumption~\eqref{eq::first_order_controllability}. Note that  $f_k(x_0) \neq 0$, since otherwise, the rank of the distribution $\spn\{f_{1}, \ldots, f_{n-1},[g,f_{k}]\}$ would be smaller than or equal to $n-1$ at this point.
Thus, $x_0$ is a transient point of $f_k$, and, according to \eg \cite{Arnold,Lee2013},  there exist $(n-1)$ differentially independent first integrals\footnote{The existence of $n-1$ independent local first integrals is equivalent to the existence of local coordinates in which $f_{k}$ is straightened out. Note that the practical computation of these coordinates is generally done by computing the associated first integrals (see again \cite{Arnold,Lee2013}). Therefore, straightening out $f_{k}$ would not provide significant simplifications in general.} of $f_k$, noted $z_{1}= \psi_{0,1}(x), \ldots, z_{n-1} = \psi_{0,n-1}(x)$, \ie satisfying:
\begin{equation}
\label{eq::first_integral}
L_{f_k}\psi_{0,i}(x) = 0, \quad i=1,\cdots,n-1,
\end{equation}
where we have denoted by $L_{\gamma}h$ the Lie derivative of an arbitrary differentiable function $h$ along the vector field $\gamma$.   

In order to show that, from these first integrals, one can deduce the diffeomorphism that puts the system in  canonical form \eqref{flat-CLdynfeed-equiv:eq}, whose Jacobian matrix is explicitly given by the matrix $N$ below (see \eqref{N'def:eq}), that we prove to be locally invertible,  we first show that the Lie derivative of the vector function $\psi_0 \triangleq (\psi_{0,1}, \cdots, \psi_{0,n-1})$ of the first integrals along the system vector field $g$ does not depend on the input $u_{k}$.
Thanks to~\eqref{eq::first_integral}, we have: 
\begin{equation}\label{dotz:eq}
\dot{z}_{i} = \sum_{j=1}^{n} \frac{\partial \psi_{0,i}}{\partial x_{j}}(x) \dot{x}_{j} 
= L_{f_0}\psi_{0,i}(x) + \sum_{j=1}^{n-1} u_jL_{f_j}\psi_{0,i}(x)
= L_{f_0}\psi_{0,i}(x) + \sum_{j \neq k} u_jL_{f_j}\psi_{0,i}(x) \triangleq \psi_{1,i}(x,u).
\end{equation}
Thus, $\frac{\partial \dot{z}_i}{\partial u_j} = \frac{\partial \psi_{1,i}}{\partial u_j}  = L_{f_j} \psi_{0,i}$ for $j \neq k$ and $\frac{\partial \psi_{1,i}}{\partial u_k} = 0$ for all $i \in \{1, \cdots, n-1\}$, thus meaning that $\psi_{1,i}$ depends only on $(x,\hat{u})$ where $\hat{u}$ denotes the collection of inputs where $u_k$ has been removed, \ie  $\hat{u} \triangleq \left( u_{1},\ldots, u_{k-1},u_{k+1},\ldots,u_{n-1} \right)$. 

Denoting by $z \triangleq (z_1, \cdots, z_{n-1}) = \psi_{0}(x)$ and $\dot{z} \triangleq \left( \psi_{1,1}(x,\hat{u}), \ldots, \psi_{1,n-1}(x,\hat{u})\right) \triangleq \psi_1(x,\hat{u})$, $D\psi_{0}(x) = \frac{\partial \psi_{0}}{\partial x}(x)$ the  Jacobian matrix of $\psi_0$ at the point $x$, of rank $n-1$ according to the independence of the first integrals, and $\hat{f} \triangleq \left( f_{1}, \ldots, f_{k-1},f_{k+1}, \ldots, f_{n-1}\right)$, 
we have that $\frac{\partial \dot{z}}{\partial \hat{u}} =\frac{\partial \psi_{1}}{\partial \hat{u}} = \left(D\psi_{0}\right) \hat{f}$ does not depend on $u$ and $\rk \left ( \frac{\partial \psi_{1}}{\partial \hat{u}}\right )  = \rk \left( \left(D\psi_{0}\right) \hat{f}\right)= n-2$ since, by \eqref{eq::first_order_controllability},  the $n-2$ columns of $\hat{f}$ are independent.

We next consider the following matrix:
$$
M \triangleq  \left( \begin{array}{cc}
\ds \frac{\partial \psi_{0}}{\partial x}&\ds 0\vspace{0.2em}\\
\ds \frac{\partial \psi_{1}}{\partial x}&\ds \frac{\partial \psi_{1}}{\partial \hat{u}}
\end{array}\right)
$$
Clearly, $M$ is the Jacobian matrix of the mapping $\psi \triangleq (\psi_{0},\psi_{1}): (x,\hat{u}) \mapsto (\psi_{0}(x) = z, \psi_{1}(x,\hat{u}) = \dot{z})$. We prove that its rank is exactly $2n-2$ in an open neighborhood of $(x_{0},u_{0})$.

Assume, on the contrary, that the columns of $M$ are linearly dependent in a given neighborhood of $(x_{0},u_{0})$. Thus, there exist smooth functions $\lambda_{1}, \ldots, \lambda_{n}$ and $\mu_{1}, \ldots, \mu_{n-1}$, not all equal to 0, such that
\begin{align}
&\sum_{j=1}^{n} \lambda_{j}(x,\hat{u})\frac{\partial \psi_{0}}{\partial x_j} (x)= 0. \label{upperleftblock:eq}\\
&\sum_{j=1}^{n} \lambda_{j}(x,\hat{u})\frac{\partial \psi_{1}}{\partial x_j} (x,\hat{u})+ \sum_{j=1, j \neq k}^{n-1} \mu_{j}(x,\hat{u}) \frac{\partial \psi_{1}}{\partial u_j} (x,\hat{u}) = 0 \label{lowerblock:eq}
\end{align}
In view of equation \eqref{upperleftblock:eq}, since $(\psi_{0,1}, \ldots, \psi_{0,n-1})$ are independent first integrals of $f_k$, we immediately deduce that there exists a smooth function $\lambda$ such that 
$$\left ( \begin{array}{c} \lambda_{1}\\\vdots\\\lambda_{n} \end{array} \right) = \lambda f_k$$
and thus, using \eqref{dotz:eq}, equation \eqref{lowerblock:eq} reads
\begin{equation}\label{dep:eq}
\lambda (x,\hat{u}) \frac{\partial \psi_{1}}{\partial x}(x,\hat{u}) f_{k}(x) + \sum_{j=1, j \neq k}^{n-1} \mu_{j}(x,\hat{u}) \frac{\partial \psi_{1}}{\partial u_j} (x) = 
\lambda L_{f_{k}}\psi_{1}(x,\hat{u}) + \sum_{j=1, j\neq k}^{n-1} \mu_{j}L_{f_{j}}\psi_{0}(x,\hat{u})
= 0
\end{equation}
but, again using \eqref{dotz:eq} and \eqref{eq::first_integral}, we get  $L_{f_{k}}\psi_1 =  L_{f_{k}}L_{g}\psi_0 = - L_{[g,f_{k}]}\psi_0$ and \eqref{dep:eq} becomes
$$
 - \lambda L_{[g,f_{k}]}\psi_0  + \sum_{j=1, j\neq k}^{n-1} \mu_{j} L_{f_{j}}\psi_0 = 
L_{\left(- \lambda [g,f_{k}] + \sum_{j=1, j\neq k}^{n-1} \mu_{j}f_{j} \right)} \psi_0
= 0
$$
which proves that the linear combination $- \lambda [g,f_{k}] + \sum_{j=1, j\neq k}^{n-1} \mu_{j}f_{j}$ of  $ f_{1}, \ldots, f_{k-1},f_{k+1}, \ldots, f_{n-1}$ and $[g,f_{k}]$ must be colinear to $f_k$, hence contradicting the assumption \eqref{eq::first_order_controllability}. We have thus proven that $\rk{M}= 2n-2$ in a neighborhood of $(x_0,u_0)$, as announced and thus that the mapping $\psi \triangleq (\psi_{0},\psi_{1}): (x,\hat{u}) \mapsto (\psi_{0}(x) = z, \psi_{1}(x,\hat{u}) = \dot{z})$ is locally invertible.

Therefore $x$ and $\hat{u}$ can be obtained from $z$ and $\dot{z}$ in a unique way
$$x=\varphi_{0}(z,\dot{z}), \quad \hat{u}= \varphi_{1}(z,\dot{z})$$
where $(\varphi_{0},\varphi_{1})$ is the inverse mapping of $\psi =(\psi_{0},\psi_{1})$.

In order to obtain the last input $u_k$ from the derivatives of $z$, and thus prove that $z$ is a flat output, we compute $\ddot{z}= L_{g}\psi_{1}$:

$$\begin{aligned}
\ddot{z}= L_g \psi_{1}
= &~L^{2}_{f_{0}}\psi_{0} +\sum_{i\neq k} \hat{u}_{i}L_{f_{0}}L_{f_{i}}\psi_{0} + \sum_{i\neq k} \hat{u}_{i}L_{f_{i}}L_{f_{0}}\psi_{0} + u_{k}L_{f_{k}}L_{f_{0}}\psi_{0} \\
&+ \sum_{i\neq k}u_{k}\hat{u}_{i}L_{f_{k}}L_{f_{i}}\psi_{0}
+ \sum_{i,j\neq k} \hat{u}_{i}\hat{u}_{j}L_{f_{i}}L_{f_{j}}\psi_{0} + \sum_{i\neq k} \dot{\hat{u}}_{i}L_{f_{i}}\psi_{0} \triangleq \psi_{2}(x,\hat{u},\dot{\hat{u}},u_k)
\end{aligned}$$

This immediately yields
\begin{equation}\label{zddotdothatublock:eq}
\frac{\partial \psi_{2}}{\partial \dot{\hat{u}}_{i}}= L_{f_{i}}\psi_{0} = \frac{\partial \psi_{1}}{\partial \hat{u}_{i}}
\end{equation}
and 
\begin{equation}\label{zddotukblock:eq}
\frac{\partial  \psi_{2}}{\partial u_{k}}=L_{f_{k}} L_{f_{0}}\psi_{0} + \sum_{i\neq k}\hat{u}_{i}L_{f_{k}}L_{f_{i}}\psi_{0} = L_{f_{k}}L_{g}\psi_{0}= - L_{[g,f_{k}]}\psi_{0} \neq 0
\end{equation}
as an immediate consequence of \eqref{eq::first_order_controllability}.

Therefore, we may complete the matrix $M$ by the following \mbox{$3(n-1)\times 3(n-1)$} square matrix
\begin{equation}\label{N'def:eq}
N \triangleq \left( \begin{array}{cccc}
\ds \frac{\partial \psi_0}{\partial x}&\ds 0&0&0 \vspace{0.3em}\\
\ds \frac{\partial \psi_1}{\partial x}&\ds \frac{\partial \psi_1}{\partial \hat{u}}&0&0\vspace{0.3em}\\
\ds \frac{\partial \psi_2}{\partial x}&\ds \frac{\partial \psi_2}{\partial \hat{u}}&\ds \frac{\partial \psi_2}{\partial u_{k}}&\ds \frac{\partial \psi_2}{\partial \dot{\hat{u}}}
\end{array}\right)
\end{equation}
which, according to \eqref{zddotdothatublock:eq} and \eqref{zddotukblock:eq}, is clearly invertible.
We indeed recognize that $N$ is the Jacobian matrix of the transformation $\ol{\psi} \triangleq (\psi_{0},\psi_{1},\psi_{2})$, which proves that $\ol{\psi}$ is a local diffeomorphism and thus that $x$ and $u= (\hat{u},u_{k})$ may be expressed as functions of $(z,\dot{z},\ddot{z})$.

Putting these results together, we have proven that $z$ is a vector of flat outputs
\end{proof}

\begin{rem}
A comparable result has been proven by P. Martin
in~\cite{Martin93} using different ideas, related to system structure
at infinity. More precisely, in~\cite{Martin93},
condition~\eqref{eq::first_order_controllability} is proven to be a
flatness sufficient condition, but the role played by first integrals
of one of the control vector fields in the construction of flat
outputs has not been brought to light.   
\end{rem}
\ck

\begin{rem}
In the case of systems with $n-1$ inputs, thanks to theorem~\ref{theo::flat_outputs}, the general approach to the computation of flat outputs presented in \cite{Levine09,Levine11}, based on an implicit representation of system \eqref{equ::explicit}, is not needed. Moreover, the present direct approach does not require an unbounded recursion as in the above mentioned references. These facts will yield important simplifications in the analysis of flatness singularities in the next sections.
\end{rem}

\begin{rem}
The last part of the proof of theorem~\ref{theo::flat_outputs} could be slightly shortened by remarking that, since $x$ has been proven to be a function of $(z,\dot{z})$, the last input $u_k$ may be obtained by differentiating $x$ with respect to time, but we have preferred a more explicit argument by constructing the Jacobian matrix of the diffeomorphism expressing $x,\hat{u},u_{k}$ in function of $(z,\dot{z},\ddot{z})$.
\end{rem}

\subsubsection{Interpretation in terms of Feedback Linearization}
\label{subsec:feedback}

Consider the $(2n-2)$-dimensional extended system (see also
~\cite{Martin93}):  
\begin{equation}\label{extsys:eq}
\begin{aligned}
&\dot{x} =  \left( f_{0}(x) + \sum_{i\neq k} u_{i} f_{i}(x) \right) + u_{k} f_{k}(x)\\
&\dot{u}_{i} = v_{i}, \quad  i \neq k
\end{aligned}
\end{equation}

with drift\footnote{hence we have $g=f+u_{k}f_{k}$.}
\begin{equation}\label{f:drift:eq}
f \triangleq f_{0} + \sum_{i \neq k} u_{i} f_{i}
\end{equation} 
and control vector fields defined by: 
$$
\tilde{f}_{i} \triangleq \left \{ \begin{array}{cl}
\ds \frac{\partial }{\partial u_{i}} , & i \in \{1,\ldots,n-1\},
~ i \neq k \\ 
\ds f_{k}, & i = k.
\end{array} \right.
$$

Then we consider the distributions 
$$G_{0} = \spn\{\tilde{f}_{1}, \ldots, \tilde{f}_{n-1}\}, \quad G_{1}
= G_{0} + \ad_{f} G_{0}$$  
in a neighborhood of a generic point, where we have denoted $\ad_{f}
G_{0} \triangleq \left\{ \ad_{f} \gamma \mid \gamma\in G_{0}\right\}$. 
A direct computation immediately shows that $[\tilde{f}_{i},\tilde{f}_{j}] = 0$ for all $i,j$, that $\ad_{f}\tilde{f}_{i} = -f_{i}$ if $i\neq k$, and 
$\ad_{f}\tilde{f}_{k} = \ad_{f}f_{k}= [f,f_{k}] = [g,f_{k}]$ (where $g$ is defined by~\eqref{gdef:eq}).  
Thus, according to
assumption \eqref{eq::first_order_controllability}, $G_{0}$ is clearly involutive and 
$$G_{1} = \spn\left\{\frac{\partial }{\partial u_{1}}, \ldots,
  \frac{\partial }{\partial u_{k-1}}, \frac{\partial }{\partial
    u_{k+1}}, \ldots, \frac{\partial }{\partial u_{n-1}},f_{1},
  \ldots, f_{n-1},[g,f_k]\right\}$$  
is involutive and has rank $2n-2$ in the neighborhood under consideration. Hence, according
to~\cite{Huntetal83,JakubczykRespondek80}, system \eqref{extsys:eq} is 
static feedback linearizable in a suitable
neighborhood of the extended state manifold of dimension $(2n-2)$ and local coordinates
$(x, \hat{u}) = \left( x_{1},\ldots, x_{n}, u_1,\ldots,
u_{k-1},u_{k+1},\ldots,u_{n-1} \right)$, where $\hat{u}
\triangleq (u_1,\ldots, u_{k-1},u_{k+1},\ldots,u_{n-1})$ is the state extension of dimension $n-2$. In other words, there
exists a regular extended-state feedback $v_i = a_i(x,\hat{u},w)$, for
$i=1, \cdots, n-1$, $i \neq k$, and $u_{k}= a_{k}(x,\hat{u},w)$, where
$w\triangleq \left( w_{1}, \ldots, w_{n-1}\right)$ is the new input,
and a local
diffeomorphism $(z, \dot{z}) = \psi(x,\hat{u})$, such that the
closed-loop system of dimension $2n-2$, namely  
$$\begin{array}{ccl}
\dot{x} & = & \ds f_{0}(x) + \sum_{i\neq k} \hat{u}_{i} f_{i}(x) +
a_{k}(x,\hat{u},w) f_{k}(x)\\ 
\dot{\hat{u}}_{i} & = & \ds a_{i}(x,\hat{u},w), \quad i \neq k
\end{array}$$
where we have renamed $u_{i} \triangleq \hat{u}_{i}$, $i\neq k$,
can be transformed by $\psi$ into the $(2n-2)$-dimensional linear controllable one
$$\ddot{z}_{i} = w_{i}, \quad i =1,\ldots,n-1.$$

\subsection{A First Atlas Construction}

Let $\Omega_{0} \subset X \times \R^{n-1}$ be the set of points
$(x,u)$ that satisfy
assumption~\eqref{eq::first_order_controllability}\footnote{Note that
  condition~\eqref{eq::first_order_controllability} depends on $u$ by
  $f=f_{0} + \sum_{i \neq k} u_{i} f_{i}$. Note furthermore that, according to \eqref{gdef:eq} and \eqref{f:drift:eq}, $[g,f_{k}]=[f,f_{k}]$.},
\ie 
\begin{equation}\label{omega:def}
\begin{aligned}
\Omega_{0} \triangleq \{ (x,u)\in &X\times \R^{n-1} \mid \exists k :
\spn\{f_{1}, \ldots, f_{n-1}, [f,f_{k}]\} = \TT_{x}X  \} 
\end{aligned}
\end{equation}
where $\TT_{x}X$ denotes the tangent space of $X$ at the point $x$ and $f=f_{0} + \sum_{i \neq k} u_{i} f_{i}$ (see \eqref{f:drift:eq}). We denote by $\tilde\Omega_{0} \subset \X$, the
set of points $(x,\overline{u})$ whose projection
$(x,u)$ in $X \times \R^{n-1}$ belongs to $\Omega_{0}$

We also consider the set
\begin{equation}
\Omega \triangleq \{ (x,u)\in X\times\R^{n-1} \mid \exists k^{\star}\in \N:  \rk\Dd_{k^{\star}}(x,u) = n  \} 
\end{equation}
with $\Dd_{k}$ defined by \eqref{Dk-seq:def}. Recall indeed that $\Omega_{0}
\subset \Omega$ (see lemma~\ref{lemma::controllability_at_first_oder} and proposition~\ref{prop:strong_contr_rk}). 
We also denote by $\tilde\Omega \subset \X$, the
set of points $(x,\overline{u})$ whose projection
$(x,u)$ in $X \times \R^{n-1}$ belongs to $\Omega$. We indeed also have $\tilde\Omega_0 \subset \tilde\Omega$.

Then the following assertion holds: 
\begin{thm}
\label{theo::atlas}
Under assumption~\eqref{eq::first_order_controllability}, for each point $(x_{0},\overline{u}_{0}) \in \tilde\Omega_0$, there exists an open
neighborhood $U_{(x_{0},u_{0})} \subset \tilde\Omega_0$ of
$(x_{0},\overline{u}_{0})$\footnote{The neighborhood $U_{(x_{0},u_{0})}$ can always be chosen of the form $U^{0}_{(x_{0},u_{0})} \times \R^{n-1}_{\infty}$ where $U^{0}_{(x_{0},u_{0})}$ is a neighborhood  of $(x_{0},u_{0})$ in $\Omega_0$ and thus only depends on $(x_{0},u_{0})$.}, and
a well-defined flat output 
$z = \Phi_{(x_{0},u_{0})}(x,u,\dot{u})$ in $U_{(x_{0},u_{0})}$, constructed according to
theorem~\ref{theo::flat_outputs}. Moreover, $(U_{(x_{0},u_{0})},\Phi_{(x_{0},u_{0})})$
constitutes a Lie-B\"acklund chart, the collection of which
defines a Lie-B\"acklund atlas.  

Moreover, the set of intrinsic singularities is contained in
$\tilde\Omega_{0}^{\CC}$ and contains $\tilde\Omega^{\CC}$, where the superscript $^{\CC}$ stands for the complementary of the corresponding set.  
\end{thm}
\begin{proof}
For each point $(x_{0},\overline{u}_{0}) \in \tilde{\Omega}_0$, there exists some $k$ and an open  neighborhood  $U_{(x_{0},u_{0})}$ of $(x_{0},\overline{u}_{0})$ such that condition~\eqref{eq::first_order_controllability} holds and such that there exists a flat output $z=\Phi_{(x_{0},u_{0})}(x,u,\dot{u})$ made of $n-1$ independent first integrals of $f_k$ in $U_{(x_{0},u_{0})}$ according to theorem~\ref{theo::flat_outputs}. The
charts, therefore made of the pairs $(U_{(x_{0},u_{0})},
\Phi_{(x_{0},u_{0})})$ are indeed 
Lie-B\"acklund charts and naturally compatible (property (ii) of
definition~\ref{L-B-charts:def}) thanks to the transitivity of the
Lie-B\"acklund equivalence relation. Therefore they form a
Lie-B\"acklund atlas.

Since $\tilde\Omega_{0}$ contains only regular points, and since, in
$\tilde\Omega$, the $\Gamma$-accessibility rank condition $\rk\Gg_{k^{\star}}(x,u) = n$ is satisfied
(see lemma~\ref{lemma::controllability_at_first_oder}) and is a necessary
flatness condition (theorem~\ref{access_rk_cond:thm}), the last assertion
of this theorem is proven. 
\end{proof}

\begin{rem}\label{0-Lebesgue:rem} 
When the fields are analytic, the complementary of $\Omega_{0}$, made of the
points where   $\dim(\spn\{f_{1}, \ldots, f_{n-1}, [f,f_{k}]\}) < n$ for all $k \in \{0,\ldots,n-1\}$ is at least of
codimension 1, hence the set of intrinsic singularities has zero
Lebesgue measure.  When the fields are $C^\infty$ but not analytic, the Lebesgue measure of
$\Omega_{0}^{\CC}$ may be non-zero.  
\end{rem}

\subsection{More on the Set of Intrinsic Singularities, Non Generic Points}
\label{sec:more}

In the previous section, we have proven that the set of intrinsic
singularities is contained in $\tilde\Omega_{0}^{\CC}$ and contains
$\tilde\Omega^{\CC}$. In this section, we investigate more in depth the
structure of $\tilde\Omega_{0}^{\CC}$ and show that there might exist points
of $\tilde\Omega_{0}^{\CC}$ where the system is still flat, or otherwise
stated, $\tilde\Omega_{0}^{\CC}$ might contain some apparent singularities.  

The next result studies degenerated situations, compared to \eqref{eq::first_order_controllability}, namely when some of the control vector fields become linearly dependent of the others. Therefore, it is intended to be applied to points $(x,u)\not\in \Omega_{0}$.

We thus consider a point $x_{0}$ and the distribution $\Gamma_{0} \triangleq
\spn\{f_1, \cdots, f_{n-1}\}$ and denote by $\Gamma_{0}(x_{0})$ the vector space generated by the vectors $\{f_1(x_{0}), \cdots, f_{n-1}(x_{0})\}$. We assume that $\dim \Gamma_{0}(x_{0})=n-p$, with $p > 1$\footnote{Recall that $\Gamma_{0}$ is assumed to have dimension $n-1$ in an open dense subset $\mathcal{O} \subset X$. Thus, we indeed have $x_{0}\not\in \mathcal{O}$}.
Without loss of generality and up to a renumbering of the $f_{i}$'s, $i=1, \ldots, n-1$, we note 
\begin{equation}\label{Gamma-a-b:def}
\Gamma_{0}^{a} \triangleq \spn\{f_{1}, \ldots, f_{n-p}\}, \quad 
\Gamma_{0}^{b} \triangleq \spn\{f_{n-p+1}, \ldots, f_{n-1}\},
\end{equation}
and assume that $\Gamma_{0}(x_{0}) = \Gamma_{0}^{a}(x_{0})$, thus meaning that the dimension of $\Gamma_{0}$ drops down from $n-1$ to $n-p$ at $x_{0}$ and that $\Gamma_{0}^{b}(x_{0}) \subset \Gamma_{0}^{a}(x_{0})$.

For simplicity's sake, we note 
\begin{equation} 
\begin{array}{lll}
\label{ua-ub-fa-fb:def}
\ds u_{a} \triangleq \left( u_{1}, \ldots, u_{n-p} \right), &\ds f_{a} \triangleq \left( f_{1},\ldots, f_{n-p} \right), &\ds u_{a}f_{a} \triangleq \sum_{i=1}^{n-p} u_{i}f_{i}, \\
\ds u_{b} \triangleq \left( u_{n-p+1}, \ldots, u_{n-1} \right), &\ds f_{b} \triangleq \left( f_{n-p+1}, \ldots, f_{n-1} \right), &\ds u_{b}f_{b} \triangleq \sum_{i=n-p+1}^{n-1}u_{i}f_{i}.
\end{array}
\end{equation}
Thus, the system equation \eqref{equ::explicit} reads 
$$\dot{x}=f_{0}(x) + u_{a}f_{a}(x) + u_{b}f_{b}(x)
$$ 
and $u_{b}f_{b}$ may now be considered as part of the drift, the independent controls being restricted to $u_a$. Therefore, we embed the drift vector field $f_{0}+u_{b}f_{b}$ in a vector field of $X\times \R^{p-1}_{\infty}$, given by
\begin{equation}\label{f0bext:def}
F_{0}^{b}(x,\ol{u}_{b}) \triangleq f_{0}(x) + u_{b}f_{b}(x)  + \tau_{b}, 
\end{equation}
where $\tau_{b}$ is the trivial Cartan field of $\R^{p-1}_{\infty}$ with global coordinates $\ol{u}_{b}\triangleq \left( u_{b},\dot{u}_{b},\ldots\right)$:
\begin{equation}\label{taub:def}
\tau_{b} \triangleq \sum_{k\geq 0}\sum_{j=n-p+1}^{n-1} u_{j}^{(k+1)}\frac{\partial}{\partial u_{j}^{(k)}} \triangleq \sum_{k\geq 0} u_{b}^{(k+1)}\frac{\partial}{\partial u_{b}^{(k)}}.
\end{equation}

We also introduce the following sequence of distributions of the tangent bundle $\TT X\times \TT \R^{n-1}_{\infty}$:
\begin{equation}\label{Gamma-a-k:def}
\Gamma_{k+1}^{a} \triangleq \Gamma_{k}^{a} + \ad_{F_{0}^{b}}\Gamma_{k}^{a}, \quad \forall k\geq 0.
\end{equation}

The following lemma shows that, in fact, the $\Gamma_{k}^{a}$'s are all contained in $\TT X$ and thus have dimension less than or equal to $n$.

\begin{lem}
We have $\Gamma_{k}^{a}(x,\ol{u}_{b}) \subset \TT_{x} X$ for every $k$, every $x\in X$ and every $\ol{u}_{b}\in \R^{p-1}_{\infty}$, and there exists $k^{\star} \in \N$ such that $\dim \Gamma_{k}^{a}(x,\ol{u}_{b}) \leq  \dim \Gamma_{k^{\star}}^{a}(x,\ol{u}_{b}) \leq n$ for every $k\in \N$ and $(x,\ol{u}_{b})\in X\times \R^{p-1}_{\infty}$.
Moreover, if $(x,\ol{u}_{b})$ is such that the $\Gamma_{k}^{a}(x,\ol{u}_{b})$'s have locally constant dimensions, then $k^{\star} \leq p$.
\end{lem}
\begin{proof}
Consider the sequence of distributions of $\TT X$:
\begin{equation}\label{Delta-a-k:def}
\Delta_{k+1}^{a} \triangleq \Delta_{k}^{a} + \ad_{f_{0}}\Delta_{k}^{a} + [\Gamma_{0}^{b},\Delta_{k}^{a}], \quad \forall k\geq 0, \qquad \Delta_{0}^{a}= \Gamma_{0}^{a}.
\end{equation}

Every $\Delta_{k}^{a}$ is a subdistribution of $\TT X$, hence $\dim\Delta_{k}^{a}(x)\leq n$ for all $k\in \N$ and all $x\in X$, with $\Delta_{k}^{a} \subset \Delta_{k+1}^{a}$. Thus, there exists an integer $k^{\star}_{\Delta}$, possibly depending on $x$, this dependence being omitted for simplicity's sake,  such that $\dim \Delta_{k}^{a}(x) \leq \dim\Delta_{k^{\star}_{\Delta}}^{a}(x)$ for all $k\leq k^{\star}_{\Delta}$ and $x\in X$, and $\dim \Delta_{k}^{a}(x) = \dim\Delta_{k^{\star}_{\Delta}}^{a}(x)$ for all $k\geq k^{\star}_{\Delta}$ and $x\in X$.

Next, considering the vector field $F_{0}^{b}$ given by \eqref{f0bext:def}, we show by induction that the distributions \eqref{Gamma-a-k:def} satisfy
$\Gamma_{k}^{a}(x,\ol{u}_{b}) \subset \Delta_{k}^{a}(x)$ for all $k\geq 0$, all $x\in X$ and all 
$\ol{u}_{b}$ in $\R^{p-1}_{\infty}$. 

This relation is indeed valid for $k=0$.

Assuming that it holds up to $j$, a vector field $\gamma \in \Gamma_{j}^{a}$ has the form
$$\gamma(x,\ol{u}_{b}) \triangleq \sum_{r=1}^{r_{j}}\alpha_{r}(x,\ol{u}_{b})\gamma_{r}(x),$$ where $\{ \gamma_{r}, r=1,\ldots, r_{j} \}$ are chosen in a basis of $\Delta_{j}^{a}$, hence independent of $\ol{u}_{b}$ and commuting with $\tau_{b}$, \ie $[ \tau_{b}, \gamma_{r}]=0$, and where $r_{j}$ stands for the dimension of $\Gamma_{j}^{a}(x,\ol{u}_{b})$. 

Let us compute $\ad_{F_{0}^{b}}\gamma$. We have
$$\begin{aligned}
\ad_{F_{0}^{b}}\gamma &= [ f_{0}+u_{b}f_{b}+\tau_{b}, \gamma ]= \ad_{f_{0}}\gamma + u_{b} [ f_{b},\gamma]+ \sum_{r=1}^{r_{j}} \left( L_{\tau_{b}}\alpha_{r} \right) \gamma_{r} \\
&\in \ad_{f_{0}}\Delta_{j} + [\Gamma_{0}^{b},\Delta_{j}] + \Delta_{j} = \Delta_{j+1}
\end{aligned}$$ 
which proves that $\Gamma_{j+1}^{a}(x,\ol{u}_{b}) \subset \Delta_{j+1}^{a}(x) \subset \TT_{x}X$ for all $x\in X$ and all 
$\ol{u}_{b}$ in $\R^{p-1}_{\infty}$. 

Since  $\Gamma_{j}^{a}(x,\ol{u}_{b}) \subset \Gamma_{j+1}^{a}(x,\ol{u}_{b}) \subset \TT_{x}X$ for all $(x,\ol{u}_{b})\in X\times \R^{p-1}_{\infty}$, we immediately conclude that $\dim \Gamma_{j}^{a}(x,\ol{u}_{b})$ is bounded above by $n$ for all $(x,\ol{u}_{b})\in X\times \R^{p-1}_{\infty}$. 
If, moreover, $(x,\ol{u}_{b})$ is such that the $\Gamma_{k}^{a}(x,\ol{u}_{b})$'s have locally constant dimensions, this bound is reached at some integer $k^{\star}\leq p$ since $\dim \Gamma_{k+1}^{a} - \dim \Gamma_{k}^{a} \geq 1$ for all $k< k^{\star}$, with $\dim \Gamma_{0}^{a} = n-p$, hence the lemma.
\end{proof}
\bigskip

From now on, we shall use the notation $\ol{u}^{(\alpha)} \triangleq \left( u, \dot{u}, \ldots, u^{(\alpha)}\right)$ for every finite $\alpha\in \N$ and $\ol{u}\in \R^{m}_{\infty}$, \ie the vector of successive derivatives of $u$ from the order 0, with $u^{(0)}\triangleq u$, up to the order $\alpha$.
\bigskip

\begin{thm}\label{flat-non-gener-thm}
Assume that, in a given neighborhood $V(x_{0})\subset X$ of $x_{0}$ and for all $\ol{u}_{b}$ in an open dense subset of $\R^{p-1}_{\infty}$:
\begin{itemize}
\item[(i)] $\Gamma_{k}^{a}$ has constant dimension and is involutive for every $k\geq 0$,
\item[(ii)] $\Gamma_{p}^{a}= \TT\R^{n}$.
\end{itemize}
Then, system~\eqref{equ::explicit} is flat with flat output $(z_{1}, \ldots, z_{n-p}, u_{n-p+1}, \ldots, u_{n-1})$, where $z_{i}\triangleq \varphi_{i,0}(x,\ol{u}_{b}^{(k_{i}-3)})$, $i=1,\ldots, n-p$, is defined in a neighborhood of every point of a dense open set of  $V(x_{0})\times \R^{(p-1)(2k_{1}-5)}$, the integers $k_1 \geq \ldots \geq k_{n-p} \geq 0$ being the list of Brunovsk\'y controllability indices (see proposition~\ref{Gamma:prop} below), and $x_{0}$ is an apparent singularity.
\end{thm}

Before stating the proof, let us sketch the ideas, that display strong similarities with the construction of the diffeomorphism that transforms the nonlinear dynamics into the Brunovsk\'y controllability canonical form of the static feedback linearization problem (see \cite[Corollary~1]{JakubczykRespondek80},  \cite[Chap.1, Sec. 3 and Chap. 5, Sec. 6]{Isidori86} or \cite[Sec. 4.1.3]{Levine97}), the main difference being that the distributions \eqref{Gamma-a-k:def} may depend on $u_{b}$ and a finite number of its derivatives. More precisely, transforming $x$ and $u_b$ and successive derivatives into a flat output $z$ implies that the Brunovsk\'y controllability indices of the corresponding linear system are obtained from \eqref{Gamma-a-k:def}. Moreover, the successive derivatives of $z$ with respect to the drift \eqref{f0bext:def}, up to a certain order, cannot depend on $u_a$ (see \eqref{z-i-k:eq} and \eqref{z-i-ki-1:eq} below). The corresponding conditions are expressed as conditions on the iterated Lie brackets \eqref{L-ad-j-f0-fk:eq} that generate the distributions \eqref{Gamma-a-k:def} and, thanks to properties (i) and (ii) of theorem~\ref{flat-non-gener-thm}, they imply the existence of $n-p$ local first integrals of the $\Gamma_{k}^{a}$'s, by Frobenius theorem (see proposition~\ref{Gamma:prop} below), first integrals that constitute the $n-p$ first components of a flat output, the $p-1$ remaining ones being the components of $u_b$.

We start with the following proposition:
\begin{prop}\label{Gamma:prop}
Assume that the assumptions (i)-(ii) of theorem~\ref{flat-non-gener-thm} are valid. Then there exist integers $k_{1},\ldots,k_{n-p}$,  with $p+1 \geq k_{1} \geq \ldots,\geq k_{n-p}\geq 0$ and $\ds \sum_{i=1}^{n-p} k_{i} = n$,
(the so-called  \emph{Brunovsk\'{y} controllability indices of the $\Gamma_{k}^{a}$'s}) and
$n$ smooth independent functions $(\varphi_{1,0},\ldots, \varphi_{1,k_{1}-1}, \ldots, \varphi_{n-p,0}, \ldots, \varphi_{n-p, k_{n-p}-1})$ defined in a neighborhood $W(\xi,\ol{\nu}_{b})$ of every $(\xi,\ol{\nu}_{b})$ in a dense subset $W$ of $V(x_{0})\times \R^{(p-1)}_{\infty}$, with
\begin{equation}\label{phi-i-j:edf}
\begin{aligned}
&\varphi_{i,j}: (x, \ol{u}_{b}^{(k_{i}+j-3)})\in W(\xi,\ol{\nu}_{b}) \rightarrow \R,  \quad j=0,\ldots, k_{i}-1,
\\
&\varphi_{i,j+1}(x, \ol{u}_{b}^{(k_{i}+j-2)})= L_{F_{0}^{b}}\varphi_{i,j} (x, \ol{u}_{b}^{(k_{i}+j-2)}),  \quad j=0,\ldots, k_{i}-2
\end{aligned}
\end{equation}
satisfying, in $W(\xi,\ol{\nu}_{b})$, and for every $(\xi,\ol{\nu}_{b})\in W$:
\begin{equation}\label{1st-integ-phi-i-j:eq} 
<d\varphi_{i,j}, \Gamma_{k}^{a}> = 0, \quad k= 0, \ldots, k_{i}-j-2, \quad j= 0, \ldots, k_{i}-2,  \quad i=1,\ldots, n-p,
\end{equation}
the $(n-p)\times (n-p)$ matrix $\Delta$ whose entries are
\begin{equation}\label{Delta:eq}
\Delta_{i,j} \triangleq L_{ad^{k_{j}-1}_{F_{0}^{b}}f_{i}}\varphi_{j,0}, \quad i,j= 1, \ldots, n-p
\end{equation}
being invertible in $W(\xi,\ol{\nu}_{b})$.

Moreover, the mapping $x \in \mathrm{pr}_{X}W(\xi,\ol{\nu}_{b}) \mapsto \Phi(x, \ol{u}_{b}^{(2k_{1}-4)})\in \R^{n}$, with $\Phi$ defined by
\begin{equation}\label{Phi:def}
\begin{aligned}
\Phi(x, \ol{u}_{b}^{(2k_{1}-4)}) &\triangleq
\left(\varphi_{1,0}(x, \ol{u}_{b}^{(k_{1}-3)}),\ldots, \varphi_{1,k_{1}-1}(x,\ol{u}_{b}^{(2k_{1}-4)}), \ldots\right.\\
& \hspace{1cm} \left. \ldots, \varphi_{n-p,0}(x,\ol{u}_{b}^{(k_{n-p}-3)}), \ldots, 
\varphi_{n-p, k_{n-p}-1}(x,\ol{u}_{b}^{(2 k_{n-p}-4)})\right)\\
&= \left(\varphi_{1,0}(x, \ol{u}_{b}^{(k_{1}-3)}),\ldots, L_{F_{0}^{b}}^{k_{1}-1}\varphi_{1,0}(x,\ol{u}_{b}^{(2k_{1}-4)}), \ldots \right.\\
&\hspace{1cm} \left. \ldots, \varphi_{n-p,0}(x,\ol{u}_{b}^{(k_{n-p}-3)}), \ldots, 
L_{F_{0}^{b}}^{k_{{n-p}}-1}\varphi_{n-p, 0}(x,\ol{u}_{b}^{(2k_{n-p}-4)})\right)
\end{aligned}
\end{equation}
is a local diffeomorphism from $\mathrm{pr}_{X}W(\xi,\ol{\nu}_{b})$ to an open subset of $\R^{n}$,  for all $\ol{\nu}_{b}$ in an open dense subset of $\R^{p-1}_{\infty}$.
\end{prop}

\begin{proof} As already announced, the proposition results from \cite[Corollary~1]{JakubczykRespondek80}. See also \cite[Chapter1, Section 3 and Chapter 5, Section 6]{Isidori86} or \cite[Section 4.1.3]{Levine97}.

We define the \emph{Brunovsk\'{y} controllability indices} $k_{1},\ldots,k_{n-p}$ associated to the $\Gamma_{k}^{a}$'s by first introducing the integers $r_{j}$ by 
$$r_{j}\triangleq \dim \Gamma_{j}^{a} -\dim \Gamma_{j-1}^{a} \quad \forall j\geq 1, \qquad r_{0} \triangleq \dim\Gamma_{0}^{a} = n-p.$$
Indeed,
$$\dim \Gamma_{j}^{a} = \sum_{k=0}^{j} r_{k} \triangleq \rho_{j}, \quad \forall j \geq 0.$$

Let us prove that $r_{j+1} \leq r_{j}$ for all $j \geq 0$.

We denote by $\{\gamma_{1}^{j},\ldots, \gamma_{\rho_{j}}^{j}\}$ a basis of the vector space $\Gamma_{j}^{a}$ for all $j\geq 0$ and $\{ \eta_{1}^{j},\ldots, \eta_{r_{j+1}}^{j}\}$ a basis of the supplementary subspace of $\Gamma_{j}^{a}$ in $\Gamma_{j+1}^{a}$.

We thus have 
$$\begin{aligned}
\Gamma_{j+1}^{a} &= \spn\{ \gamma_{1}^{j+1}, \ldots, \gamma_{\rho_{j+1}}^{j+1} \} 
\\
& = \spn\{ \gamma_{1}^{j}, \ldots, \gamma_{\rho_{j}}^{j} \} \oplus \spn\{ \eta_{1}^{j},\ldots, \eta_{r_{j+1}}^{j} \} \\
& = \spn\{ \gamma_{1}^{j-1}, \ldots, \gamma_{\rho_{j-1}}^{j-1} \} \oplus \spn\{ \eta_{1}^{j-1},\ldots, \eta_{r_{j}}^{j-1} \} \oplus \spn\{ \eta_{1}^{j},\ldots, \eta_{r_{j+1}}^{j} \}
\end{aligned}$$ 

Therefore we have: 
$\spn\{ \eta_{1}^{j},\ldots, \eta_{r_{j+1}}^{j} \} \subset \spn\{ \ad_{F_{0}^{b}}\eta_{1}^{j-1}, \ldots, 
\ad_{F_{0}^{b}}\eta_{r_{j}}^{j-1} \}$ 
which immediately yields: $r_{j+1} \leq r_j$.

The sequence $r_{j}, j\geq 0$ being non increasing with $r_{j}\geq 0$ for all $j\geq 0$, there exists an ultimate $j^{\star}$ such that $r_{j^{\star}}>0$ and $r_j=0$ for all $j > j^{\star}$. We indeed have $j^{\star}\leq p$ and $\sum_{j=0}^{j^{\star}}r_j = \dim \Gamma^{a}_{j^{\star}}  =  \Gamma^{a}_{p} = n$ by assumption (ii) of theorem~\ref{flat-non-gener-thm}.

Then, the \emph{Brunovsk\'{y} controllability indices} $k_{1},\ldots,k_{n-p}$ are given by
$$k_{i}\triangleq \# \{j \mid r_{j} \geq i\}, \quad i\geq 1$$
where $\#A$ denotes the number of elements of an arbitrary set $A$.

Again following the same lines as in \cite{JakubczykRespondek80,Isidori86,Levine97}, we immediately get that $k_{i} \leq k_{1}= j^{\star}+1\leq p +1$ for all $i$, or
$$\max\{ k_{i} \mid i \geq 0 \}= k_{1} \leq p +1.$$

Moreover, for all $i\geq 1$, the dimension jumps $r_{j}$, sorted by jumps of equal dimension $i$, being $k_{i}-k_{i+1}$, we get $\dim\Gamma^{a}_{p}=\sum_{i=1}^{n-p} i(k_{i}-k_{i+1})= k_{1}+\ldots+k_{n-p}$ and, since $\Gamma_{p}^{a}= \Gamma_{k_{1}}^{a}= \TT\R^{n}$ by assumption (ii), that $k_{1}+\ldots+k_{n-p}  = n$. 

Moreover, possibly up to a new renumbering of the $f_{i}$'s and for all $\ol{u}_{b}$ in a dense subset of $\R^{p-1}_{\infty}$, they satisfy:

\begin{itemize}
\item if $0\leq i\leq k_{n-p}-1$:
\begin{equation}\label{Gamma-i-1:def}
\Gamma_{i}^{a} = \spn\{f_{1}, \ldots,\ad_{F_{0}^{b}}^{i} f_{1}, \ldots ,
f_{n-p} , \ldots, \ad_{F_{0}^{b}}^{i} f_{n-p}\}
\end{equation}
\item if $k_{j+1}\leq i \leq k_{j}-1$, $1\leq j\leq n-p-1$:
\begin{equation}\label{Gamma-i-2:def}\begin{aligned}
\Gamma_{i}^{a}= \spn\{&f_{1}, \ldots,\ad_{F_{0}^{b}}^{i}f_{1}, \ldots, f_{j}, \ldots,\ad_{F_{0}^{b}}^{i} f_{j},\ldots \\
&\ldots, f_{j+1}, \ldots,\ad_{F_{0}^{b}}^{k_{j+1}-1} f_{j+1} , \ldots, f_{n-p}, \ldots, \ad_{F_{0}^{b}}^{k_{n-p}-1} f_{n-p}\}
\end{aligned}
\end{equation}
\item and for $i\geq k_{1}$:
\begin{equation}\label{Gamma-i-3:def}
\Gamma_{i}^{a}=\Gamma_{k_{1}}^{a} = \spn\{f_{1}, \ldots,\ad_{F_{0}^{b}}^{k_{1}-1}f_{1}, \ldots, 
f_{n-p}, \ldots, \ad_{F_{0}^{b}}^{k_{n-p}-1} f_{n-p}\}.
\end{equation}
\end{itemize}

Note, moreover, that $\Gamma_{k_{1}}^{a}$ so defined depends at most on the
$$\max\{ k_{n-p}-2, \, k_{n-p-1}-2, \ldots, k_{1}-2 \} = k_{1} - 2$$ 
first successive derivatives of  $\ol{u}_{b}$, \ie on $\ol{u}^{(k_{1}-2)}$.

 Hence, since all the $\Gamma_{i}^{a}$'s are involutive by assumption (i), by Frobenius theorem, there exist $n-p$ independent first integrals
$\varphi_{1,0}(x,\ol{u}_{b}^{(k_{1}-3)})=z_{1},\ldots, \varphi_{n-p,0}(x,\ol{u}_{b}^{(k_{n-p}-3)})=z_{n-p}$ 
defined in a neighborhood $W(\xi,\ol{\nu}_{b})$ of every $(\xi,\ol{\nu}_{b})$ in a dense subset $W$ of $V(x_{0})\times \R^{(p-1)}_{\infty}$ satisfying
\begin{equation}\label{adf0-j:eq}
L_{ad^{j}_{F_{0}^{b}}f_{k}}\varphi_{i,0}=0, \quad  j=0, \ldots, k_{i}-2, \quad i, k= 1,\ldots, n-p,
\end{equation}
or equivalently \eqref{1st-integ-phi-i-j:eq} and for which the matrix $\Delta$ given by \eqref{Delta:eq}, is invertible for all $\ol{u}_{b}$ in a dense subset of $\R^{p-1}_{\infty}$. 

Next, considering the functions defined by \eqref{phi-i-j:edf}:
$$\varphi_{i,j} (x, \ol{u}_{b}^{}) \triangleq L^{j}_{F_{0}^{b}}\varphi_{i,0}(x) , \quad  j= 1, \ldots, k_{i}-2, \quad i= 1, \ldots, n-p, $$ 
by \eqref{Gamma-i-1:def}-\eqref{Gamma-i-2:def}-\eqref{Gamma-i-3:def}, they are such that
\begin{equation}\label{adf0-j-phi-l:eq}
L_{ad^{j-l}_{F_{0}^{b}}f_{k}}\varphi_{i,l}=0, \quad  0\le l\le
  j\le k_{i}-2, \quad i, k= 1,\ldots, n-p.
\end{equation}
We can prove again, as in \cite{JakubczykRespondek80,Isidori86,Levine97}, that the mapping 
\begin{equation}\label{diffeo-thm3:eq}
\begin{aligned}
x \mapsto &\left ( \varphi_{1,0}, \ldots, \varphi_{1, k_{1}-1}, \ldots, \varphi_{n-p,0}, \ldots, \varphi_{n-p, k_{n-p}-1}\right )\\
&= (z_{1},\ldots, z^{(k_{1}-1)}_{1}, \ldots, z_{n-p},\ldots, z^{(k_{n-p}-1)}_{n-p})
\end{aligned}
\end{equation}
where $z_{i}^{(j)} = \varphi_{i,j} (x, \ol{u}_{b}^{(k_{i}+j-3)}) = L^{j}_{F_{0}^{b}}\varphi_{i,0}(x, \ol{u}_{b}^{(k_{i}+j-3)})$, is a local diffeomorphism of $X$ for every $\ol{u}_{b}$ in a dense open subset of $\R^{p-1}_{\infty}$. 
Moreover, as a consequence of \eqref{adf0-j:eq}, it is easy to prove by induction that, for all  $j=0, \ldots, k_{i}-2$, $i, k= 1,\ldots, n-p$ and $r= 0,\ldots,j$,
\begin{equation}\label{L-ad-j-f0-fk:eq}
L_{ad^{j}_{F_{0}^{b}}f_{k}}\varphi_{i,0}=0 = (-1)^{j}L_{f_{k}}L_{F_{0}^{b}}^{j}\varphi_{i,0} = (-1)^{j}L_{f_{k}}L_{F_{0}^{b}}^{j-r}\varphi_{i,r}.
\end{equation}
\end{proof}

\begin{proof}[Proof of theorem~\ref{flat-non-gener-thm}]
Assume that (i) and (ii) hold true in a neighborhood $V(x_{0})$ of $x_{0}$. Then, according to Propositon~\ref{Gamma:prop}, there exist $n$ smooth independent functions $(\varphi_{1,0},\ldots, \varphi_{1,k_{1}-1}, \ldots, \varphi_{n-p,0}, \ldots, \varphi_{n-p, k_{n-p}-1})$ that satisfy \eqref{1st-integ-phi-i-j:eq} with \eqref{Delta:eq} and such that the mapping $\Phi$ defined by \eqref{Phi:def} is a local diffeomorphism.

Differentiating $k$ times $z_{i}(t) \triangleq \varphi_{i,0}(x(t))$, for $i=1,\ldots, n-p$, with respect to time, where $x(t)$ is an integral curve of system~\eqref{equ::explicit}, we prove by induction, thanks to \eqref{adf0-j:eq}-\eqref{adf0-j-phi-l:eq}, that, for all $k=0, \ldots, k_{i}-2$, 
\begin{equation}\label{z-i-k:eq}
z_{i}^{(k)}(t) = L_{F_{0}^{b}}^{k}\varphi_{i,0} + u_{a}L_{f_{a}}L_{F_{0}^{b}}^{k-1}\varphi_{i,0}
= L_{F_{0}^{b}}^{k}\varphi_{i,0} 
\end{equation}
and
\begin{equation}\label{z-i-ki-1:eq}
z_{i}^{(k_{i}-1)}(t)= L_{F_{0}^{b}}^{k_{i}-1}\varphi_{i,0} + u_{a}L_{f_{a}}L_{F_{0}^{b}}^{k_{i}-2}\varphi_{i,0}.
\end{equation}
Thanks to \eqref{Delta:eq}, the latter relation allows to obtain $u_{a}$ as a function of $z= (z_{1},\ldots, z_{n-p})$ and $u_{b}$ and derivatives up to $k_{1}-2$. Moreover, according to \eqref{diffeo-thm3:eq}, this proves that $x$ and $u_{a}$ can be expressed as functions of the pair $(z, u_{b})$ and successive derivatives in finite number, hence the flatness property.
\end{proof}

\begin{rem}\label{rem:linearization}
As for theorem~\ref{theo::flat_outputs}, the reader may easily check \apost that theorem~\ref{flat-non-gener-thm} may be interpreted in terms of feedback linearization by extending the state as $(x,u_{b},\ldots, u_{b}^{(2k_{1}-5)})$ and inputs $(u_{a},u_{b}^{(2k_{1}-4)})$. 
However, if we try to directly apply the feedback linearization theorem by finding the extension length by essay-error, not only the number $2k_{1}-5$ is not intuitive, and may be large, but its relation with the largest dimension jump, $k_{1}$, of the sequence of distributions \eqref{Gamma-a-k:def} would be ignored.
\end{rem}

\begin{rem}\label{rem:m}
It must be noted that theorem~\ref{flat-non-gener-thm} remains valid in the case of $m\leq n-1$ control inputs, with $n-m\leq p$. This trivial verification is left to the reader.
\end{rem}

\begin{rem}
To the authors knowledge, theorem~\ref{flat-non-gener-thm} is the first one giving a practical condition for which the degenerating directions at a given point, producing a rank drop at this point, are replaced by those generated by another tower of constant rank and involutive subdistributions, namely the $\Gamma_{k}^{a}$'s. In contrast, the apparent singularities dealt with in \cite{CE14,CE17,KLO17} only concern singularities of the parameterization and are not related to a singularity of the distribution of control vector fields.
\end{rem}

\begin{rem}
The application to the global or semi-global motion planning, as
presented in \cite{KLO17} for the non-holonomic car, may be indeed
easily adapted to the present context. 
However, though important in applications, we do not shed new light on
this topic. Therefore, this aspect is not developed again here. 
\end{rem}

\section{Examples}\label{sec:ex}

We now give three simple academic examples, with computational aspects as simple as possible, to illustrate our previous results and emphasize the role played by suitably chosen local first integrals in the analysis of flatness singularities. This is why, apart from the first part of example 2, the flat outputs could also be easily determined by inspection. However, we do not want to make the reader believe that the computational complexity of our conditions is low. It is in fact comparable to the one of the already cited static feedback linearization conditions \cite{Huntetal83,JakubczykRespondek80}.

\subsection{Example 1}
\label{ex:one}
Consider the following system:

$$
\left\{
\begin{array}{ccl}
\dot{x}_1 & = & x_1 u_1 + x_2 \\
\dot{x}_2 & = & x_3 \\ 
\dot{x}_3 & = & u_2.
\end{array}\right.
$$ 

We have: $f_{0} = x_{2} \frac{\partial}{\partial x_{1}} + x_{3}
\frac{\partial}{\partial x_{2}}$, $f_{1} = x_{1}
\frac{\partial}{\partial x_{1}}$ and $f_{2} = \frac{\partial}{\partial
  x_{3}}$. When $x_{1} \neq 0$, the fields $f_{1}$ and $f_{2}$ are
linearly independent and
condition~\eqref{eq::first_order_controllability} holds since
$[g,f_{2}] = [f_{0},f_{2}] = - \frac{\partial}{\partial x_{2}}$. Therefore by
theorem~\ref{theo::flat_outputs}, the system is flat in the dense
open set $\R^{3}\setminus \{x_{1}=0\}$. A flat output is given by two independent first integrals of
$f_{2}$, e.g. $y_{1} \triangleq x_{1}$, $y_{2}\triangleq x_{2})$,
which is easily confirmed by the formulas $x_{1}=y_{1}$,
$x_{2}=y_{2}$, $x_{3}=\dot{y}_{2}$, $u_{1} =
\frac{\dot{y}_{1}-y_{2}}{y_{1}}$ and $u_{2} = \ddot{y}_{2}$. One can
also easily verify that, after extending the system by adding an
integrator to $u_{1}$, \ie $\dot{u}_{1}=v_{1}$, the extended system is
feedback linearizable. 

At a point $(x_{1},x_{2},u_{1})$, where $x_{1} = 0$, the system is
degenerated but still flat. Since $f_{1}=0$ at this point, we have
$\Gamma_{0}= \spn \{f_{2}\}$, with $\dim \Gamma_{0}= 1$. We thus apply
theorem~\ref{flat-non-gener-thm} with $p=2$, $\Gamma_{0}^{a}= \spn\{f_{2}\}$ and 
$\Gamma_{0}^{b}= \spn\{f_{1}\}$. Here $u_{b}=u_{1}$ and
$F_{0}^{b}(x,\ol{u}_{b})= f_{0} + u_{1} f_{1} + \tau_{1} = (x_{2}+u_{1}x_{1})\frac{\partial}{\partial x_{1}} + x_{3}\frac{\partial}{\partial x_{2}} + \sum_{k\geq 0} u_{1}^{(k+1)}\frac{\partial}{\partial u_{1}^{(k)}}$. The reader may easily check that 
$\Gamma_{0}^{a}=  \spn \{\frac{\partial}{\partial x_{3}} \} = \ol{\Gamma_{0}^{a}}$, 
$\Gamma_{1}^{a}=  \spn \{\frac{\partial}{\partial x_{3}}, \frac{\partial}{\partial x_{2}} \} = \ol{\Gamma_{1}^{a}}$ and 
$\Gamma_{2}^{a}=  \spn \{\frac{\partial}{\partial x_{3}}, \frac{\partial}{\partial x_{2}}, \frac{\partial}{\partial x_{1}} \}= \TT\R^{3}$.
Hence we conclude that the system is flat
with $\tilde{y}_{1} \triangleq x_{1}$ and $\tilde{y}_{2} \triangleq
u_1$ as flat output, which indeed implies that $x_{1} = 0$ is an apparent singularity.

\subsection{Example 2}
\label{ex:two}
We now consider the following 4-dimensional driftless system with 3 control inputs:

\begin{equation}\label{ex2-sys:eq}
\left \{
\begin{array}{ccl}
\dot{x}_{1} & = & u_{1} \\
\dot{x}_{2} & = & x_{3} u_{1} \\ 
\dot{x}_{3} & = & x_{4} u_{1} + x_{1} u_{3} \\
\dot{x}_{4} & = & u_{2}.
\end{array} \right.
\end{equation}

The drift is thus $f_{0}= 0$ and the control fields are respectively given by 
$$f_{1}= \frac{\partial}{\partial x_{1}} +  x_{3} \frac{\partial}{\partial x_{2}} +
x_{4} \frac{\partial}{\partial x_{3}}, \quad f_{2} = \frac{\partial}{\partial x_{4}}, \quad f_{3} =
x_{1} \frac{\partial}{\partial x_{3}}.$$  

At points where $x_{1} \neq 0$ and $u_{1} \neq 0$, these three vector fields are linearly
independent and $[g,f_{3}]= u_{1}[f_{1},f_{3}] = u_{1}\left( \frac{\partial}{\partial x_{3}} -
x_{1}\frac{\partial}{\partial x_{2}}\right) \not \in
\spn(f_{1},f_{2},f_{3})$. Therefore assumption
\eqref{eq::first_order_controllability} is satisfied and, by
theorem~\ref{theo::flat_outputs}, the system is flat with flat outputs
given by three independent first integrals of $f_{1}$. An easy
computation yields the following flat output: 
$$
z_{1} = 2x_{2}x_{4}-x_{3}^{2}, \quad z_{2}=x_{1}x_{4}-x_{3}, \quad z_{3}=x_{4}.$$
One could
instead take three independent first integrals of $f_{3}$ as well to
get 
$$z'_{1} = x_{1},\quad  z'_{2} = x_{2}, \quad  z'_{3} = x_{4}.$$   

Note that at points $x_{1} = 0$, the distribution $\spn \{f_{1},f_{2},f_{3}\}$ has dimension 2. The reader can verify that $\Gamma_{0}^{a}= \spn\{f_{1},f_{2}\}$ is not involutive, and therefore that theorem~\ref{flat-non-gener-thm} does not apply. Nevertheless, $x_{1}$, $x_{2}$, and $u_{3}$, is a flat output whenever $u_{1}\neq0$.

\subsection{Example 3}
\label{ex:three}
We now consider an extension of system~\eqref{ex2-sys:eq} by adding two integrators to $u_{1}$, which makes the system a 6 dimensional one with 3 inputs:
\begin{equation}\label{ex3-sys:eq}
\left \{
\begin{array}{ccl}
\dot{x}_{1} & = & x_{5} \\
\dot{x}_{2} & = & x_{3} x_{5} \\ 
\dot{x}_{3} & = & x_{4} x_{5} + x_{1} u_{3} \\
\dot{x}_{4} & = & u_{2}\\
\dot{x}_{5} & = & x_{6}\\
\dot{x}_{6} & = & u_{1}.
\end{array} \right.
\end{equation}
System~\eqref{ex3-sys:eq} is indeed Lie-B\"acklund equivalent to \eqref{ex2-sys:eq} and therefore flat at points where $x_{1} \neq 0$ and $u_{1} \neq 0$.

We have 
$$f_{0}= x_{5}\frac{\partial}{\partial x_{1}} + x_{3}x_{5}\frac{\partial}{\partial x_{2}} + x_{4} x_{5}\frac{\partial}{\partial x_{3}} + x_{6}\frac{\partial}{\partial x_{5}}, \quad 
f_{1}=\frac{\partial}{\partial x_{6}}, \quad f_{2} = \frac{\partial}{\partial x_{4}}, \quad f_{3} =x_{1}\frac{\partial}{\partial x_{3}}.$$ 
We now show that theorem~\ref{flat-non-gener-thm}, for $m=3$ and $n=6$
(see remark~\ref{rem:m}), applies in the open set defined by $x_{5}\neq0$, without restriction on $x_{1}$, even if the distribution 
$\spn \{f_{1},f_{2},f_{3}\}= \spn \{\frac{\partial}{\partial x_{6}}, \frac{\partial}{\partial x_{4}}, x_{1}\frac{\partial}{\partial x_{3}} \}$ 
degenerates at points such that $x_{1}=0$. We have:
$$\Gamma_{0}^{a}= \spn \{f_{1}, f_{2}\}, \quad \Gamma_{0}^{b}=  \spn \{f_{3} \}, \quad p=4 \geq n-m= 3,$$ and
$$F_{0}^{b}= f_{0} + u_{3}f_{3} + \tau_{3} = x_{5}\frac{\partial}{\partial x_{1}} + (x_{4}x_{5} +
u_{3} x_{1}) \frac{\partial}{\partial x_{3}} + x_{6}\frac{\partial}{\partial x_{5}} + \sum_{k\geq 0} u_{3}^{(k+1)}\frac{\partial}{\partial u_{3}^{(k)}}.$$

In the open set defined by $x_{5}\neq0$,  we have $\Gamma_{0}^{a}= \spn\{\frac{\partial}{\partial x_{6}},\frac{\partial}{\partial x_{4}}\}$,
$\Gamma_{1}^{a}= \spn\{\frac{\partial}{\partial x_{6}}, \frac{\partial}{\partial x_{5}}, \frac{\partial}{\partial x_{4}}, x_{5}\frac{\partial}{\partial x_{3}}\}$, and $\Gamma_{2}^{a}= \TT\R^{6}$, which proves that the assumptions of theorem~\ref{flat-non-gener-thm} are satisfied. We immediately get the flat output $z_{1}=x_{1}$, $z_{2}=x_{2}$ and $z_{3}= u_{3}$.

\section{Conclusions}\label{sec:concl}

In this paper, we have studied the set of intrinsic singularities of control affine flat systems with one input less than the number of states. We have proven two theorems in this context, showing how to construct flat outputs and their Lie-B\"{a}cklund atlases, thus allowing to deduce some inclusions of their associated set of intrinsic singularities.
We also give three examples which may be interpreted in a potentially interesting way for applications: if the system degenerates at a point but is still flat there, and if the degeneracy point corresponds to some damaged state, \eg loss of a motor or of a wing of an aircraft, then the present analysis may be helpful for input reconfiguration of the damaged system and emergency motion planning. This idea will be more thoroughly studied in a forthcoming work of the authors.

  \bigskip

\section{Acknowledgement}
The authors wish to express their warmest thanks to one of the anonymous reviewers for his/her careful reading and most useful remarks.

\bibliographystyle{amsplain}
\bibliography{IJRNCsing}

\providecommand{\bysame}{\leavevmode\hbox to3em{\hrulefill}\thinspace}
\providecommand{\MR}{\relax\ifhmode\unskip\space\fi MR }
\providecommand{\MRhref}[2]{%
  \href{http://www.ams.org/mathscinet-getitem?mr=#1}{#2}
}
\providecommand{\href}[2]{#2}
\begin{thebibliography}{10}

\bibitem{Arnold}
V.I. Arnol{'}d, \emph{Ordinary differential equations}, Springer Verlag, Berlin
  Heidelberg, 1992.

\bibitem{CE14}
D.E. Chang and Y.~Eun, \emph{Construction of an atlas for global flatness-based
  parameterization and dynamic feedback linearization of quadcopter dynamics},
  Proc. 53rd IEEE Conference on Decision and Control, IEEE, 2014, pp.~686--691.

\bibitem{CE17}
\bysame, \emph{Global chartwise feedback linearization of the quadcopter with a
  thrust positivity preserving dynamic extension}, IEEE Trans. Automat. Control
  \textbf{62} (2017), no.~9, 4747--4752.

\bibitem{FLMR95}
M.~Fliess, J.~L\'{e}vine, Ph. Martin, and P.~Rouchon, \emph{Flatness and defect
  of non-linear systems: introduction theory and examples}, Int. Journal of
  Control \textbf{61} (1995), no.~6, 1327--1361.

\bibitem{FLMR99}
\bysame, \emph{A {L}ie-{B}\"{a}cklund approach to equivalence and flatness of
  nonlinear systems}, IEEE Trans. Automatic Control \textbf{44} (1999), no.~5,
  922--937.

\bibitem{HK77}
R.~Hermann and A.J. Krener, \emph{Nonlinear controllability and observability},
  IEEE Trans. Automatic Control \textbf{22} (1977), 728--740.

\bibitem{Huntetal83}
L.~Hunt, R.~Su, and G.~Meyer, \emph{Design of multi-input nonlinear systems},
  Differential Geometric Control Theory (R.W. Brockett, R.S. Millman, and H.J.
  Sussmann, eds.), Michigan Technological University, Birkh\"{a}user, Boston,
  1983, pp.~268--298.

\bibitem{Isidori86}
A.~Isidori, \emph{Nonlinear control systems}, 1st ed., Springer-Verlag, New
  York, 1986.

\bibitem{JakubczykRespondek80}
B.~Jakubczyk and W.~Respondek, \emph{On linearization of control systems},
  Bull. Acad. Pol. Sci. Ser. Sci. Math. \textbf{28} (1980), no.~9--10,
  517--522.

\bibitem{JS72}
V.~Jurdjevic and H.J. Sussmann, \emph{Controllability of nonlinear systems},
  Journal of Differential Equations \textbf{12} (1972), 95--116.

\bibitem{K}
T.~Kailath, \emph{Linear systems}, Prentice Hall, Englewood Cliffs, NJ, 1980.

\bibitem{KLO17}
Y.~Kaminski, J.~L\'evine, and F.~Ollivier, \emph{Intrinsic and apparent
  singularities in differentially flat systems, and application to global
  motion planning}, Systems \& Control Letters \textbf{113} (2018), 117--124.

\bibitem{Lee2013}
J.~Lee, \emph{Introduction to smooth manifolds}, Springer, 2013.

\bibitem{Levine97}
J.~L{\'e}vine, \emph{Static and dynamic state feedback linearization}, vol.~3,
  ch.~4, pp.~93--126, Chapman \& Hall, 1997.

\bibitem{Levine09}
J.~L\'{e}vine, \emph{Analysis and control of nonlinear systems: A
  flatness-based approach}, Mathematical Engineering, Springer, 2009.

\bibitem{Levine11}
\bysame, \emph{On necessary and sufficient conditions for differential
  flatness}, Applicable Algebra in Engineering, Communication and Computing
  \textbf{22} (2011), no.~1, 47--90.

\bibitem{MartinThese}
Ph. Martin, \emph{Contribution \`{a} l'\'{e}tude des syst\`{e}mes
  diff\'{e}rentiellement plats}, Ph.D. thesis, Ecole Nationale Sup\'{e}rieure
  des Mines de Paris, Paris, France, 1992.

\bibitem{Martin93}
\bysame, \emph{A geometric sufficient condition for flatness of systems with
  $m$ inputs and $m+1$ states}, Proc. 32nd IEEE Conference on Decision and
  Control (San Antonio, Texas), IEEE, 1993, pp.~3431--3436.

\bibitem{NicolauThese}
F.~Nicolau, \emph{Geometry and flatness of control systems of minimal
  differential weight}, Ph.D. thesis, INSA, Rouen, Dec. 2014.

\bibitem{Nic-Resp-siam2017}
F.~Nicolau and W.~Respondek, \emph{Flatness of multi-input control affine
  systems linearizable via one-fold prolongation}, SIAM J. Control Optim
  \textbf{55} (2017), no.~5, 3171--3203.

\bibitem{Resp-2003}
W~Respondek, \emph{Symmetries and minimal flat outputs of nonlinear control
  systems}, Lecture Notes in Control and Inform. Sci., vol. Vol. 295,
  pp.~65--86, Springer, New York, 2003.

\bibitem{SM67}
L.~M. Silverman and H.~E. Meadows, \emph{Controllability and observability in
  time-variable linear systems}, SIAM Journal on Control \textbf{5} (1967),
  no.~1, 64--73.

\end{thebibliography}

\end{document}